\documentclass[12pt,a4paper]{article}

\usepackage[utf8]{inputenc}

\usepackage[english]{babel}

\usepackage{graphicx}
\usepackage{placeins} 
\usepackage{subfig}
\usepackage{wrapfig}

\usepackage{amsmath}
\usepackage{amsthm}
\usepackage{amsfonts}

\theoremstyle{plain}
\newtheorem{theorem}{Theorem}

\newtheorem{corollary}[theorem]{Corollary}
\newtheorem{lemma}[theorem]{Lemma}
\theoremstyle{definition}

\def\dd{{\rm d}}
\def\e{{\rm e}}
\def\<{\langle}\def\>{\rangle}

\DeclareMathOperator\RES{Res}

\begin{document}

\title{Rigorous numerical
verification of uniqueness and smoothness in a surface growth model}
\author{Dirk Blömker, Christian Nolde, James Robinson}

\maketitle

\begin{abstract}
Based on numerical data and a-posteriori analysis we
verify rigorously the uniqueness and smoothness of global
solutions to a scalar surface growth model with striking similarities
to the 3D Navier--Stokes equations, for certain initial data for which
analytical approaches fail.
The key point is the derivation of a scalar ODE controlling
the norm of the solution, whose coefficients
depend on the numerical data.
Instead of solving this ODE explicitly, we explore three different
numerical methods that provide rigorous upper bounds for its solution.
\end{abstract}



\section{Introduction}
We consider the following surface growth equation for
the height $u(t,x)\in\mathbb{R}$ at time $t>0$
over a point $x\in[0,2\pi]$
\begin{align}
	 u_t = - u_{xxxx} - ({u_x}^2)_{xx}
		\qquad x \in [0,2\pi], \; t \in [0,T] \label{eqn:sfg}
\end{align}
with periodic boundary conditions and subject to a moving frame,
which yields the zero-average condition 
$ \int_0^{2\pi}u(x,t)\;\mathrm{d}x = 0$.

This equation, usually with additional noise terms,
was introduced as a phenomenological model for
the growth of amorphous surfaces \cite{Si-Pl:94b,Ra-Li-Ha:00a}, 
and was also used
to describe sputtering processes \cite{Cu-Va-Ga:05}; 
see \cite{blomkerromito12} for a detailed list of references.
Based on the papers
\cite{robinsonetal07, dashtirobinson08, robinsonrubiosadowski13}
which develop the theory of `numerical verification of regularity'
for the 3D Navier--Stokes equations,
our aim here is to establish and implement numerical algorithms
to  prove rigorously  global existence and uniqueness
of solutions of (\ref{eqn:sfg}).

Despite being scalar the
equation has surprising similarities to 3D Navier--Stokes equations
\cite{blomkerflandoliromito09, blomkerromito09, blomkerromito12}.
It allows for a global energy estimate in $L^2$ and uniqueness
of smooth local solutions for initial conditions in a critical
Besov-type space that contains $C^0$ and $H^{1/2}$,
see \cite{blomkerromito12}
(similar results for the 3D Navier--Stokes equations can
be found in \cite{KocTat01}).
Here we focus on the one-dimensional model, since in this case 
more efficient numerical methods are available, and the calculations 
would be significantly slower in higher dimension.
Moreover, for the two-dimensional case the situation of 
energy estimates seems even worse,
as global existence could only be established in $H^{-1}$ using the
non-standard energy $\int_0^{2\pi} \e^{u(x)}\,\dd x$,
see \cite{MW:11} for details. Nevertheless, we believe that it should 
be possible to treat the 2D case using similar methods, 
but the analysis becomes more delicate since in two dimensions $H^1$ 
is the critical space (see \cite{blomkerromito09, blomkerromito12}).

Rigorous methods for proving numerically the existence of
solutions for PDEs are a recent and active research field.
In addition to the approach taken here there are methods based
on topological arguments
like the Conley index, see
\cite{paapeetal08, daylessardmischaikov07, zgliczynski10}, for example.
For solutions of elliptic PDEs there are methods using Brouwer's fixed-point
theorem, as discussed in the review article \cite{Pl:08} and
the references therein.

Our approach is based on \cite{robinsonetal07}
and similar to the method proposed in \cite{morosipizzocchero12}.
The key point is the derivation of a scalar ODE for
the $H^1$-norm of the difference of an arbitrary approximation, 
that satisfies the boundary conditions, to the
solution. The coefficients of this ODE
depend only on the numerical data (or any other approximation used).
As long as the solution of the ODE stays finite, one can
rely on the continuation property of unique local solutions,
and thus have a smooth unique solution up to a blowup time of the ODE.
A similar approach using an integral equation
based on the mild formulation was proposed in
\cite{morosipizzocchero08, morosipizzocchero11}.

In order to establish a bound on the blow-up time for the ODE,
one can either proceed analytically or numerically. We propose
two analytical methods: one, based on the standard Gronwall Lemma,
enforces a `small data' hypothesis and adds little to standard
analytical existence proofs. The second is based on an explicit
analytical upper bound to the ODE solution. A variant of this, a
hybrid method in which one applies an analytical upper bound on a
succession of small intervals of length $h>0$ to the numerical
solution and then restarts the argument, appears the most promising,
and a formal calculation indicates that the upper bound from the
third method in the limit of step-size to zero
converges to the solution of the ODE.

In order to derive the ODE for the $H^1$-error,
we use standard a-priori estimates. While the stability of 
the linear term $-u_{xxxx}$ means that these `worst case' 
estimates are still sufficient, an interesting alternative 
approach in a slightly different context is proposed in
\cite{nakaohashimoto09, nakaokinoshitakinura12}, where 
the spectrum of the linearized operator 
(here $Lv=-v_{xxxx} + (v_x\varphi_x)_{xx}$, where $\varphi$ 
is some given numerical data) is analysed
with a rigorous numerical method, which in the case of an
unstable linear operator yields substantially better results,
at the price of a significantly higher computational time.
This will be the subject of future research.

The paper is organized as follows.
In Section \ref{sec:apa} we establish the a-priori
estimates for the $H^1$-error between solutions
and the numerical data, which in the end gives an ODE depending on the
numerical data only.
Section \ref{sec:ode} provides the ODE estimates necessary for our
three methods, while Section \ref{sec:res} states the main results.
In the final Section \ref{sec:num}, we compare our methods using
numerical experiments.

\section{A-priori analysis}
\label{sec:apa}

In this section we establish upper bounds for the $H^1$-norm of the error
\[ d(x,t) := u(x,t) - \varphi(x,t),\]
where $u$ is a solution to our surface growth equation (\ref{eqn:sfg})
and $\varphi$ is any arbitrary, but sufficiently smooth approximation, 
that satisfies the boundary conditions.
Since we know $\varphi$, if we can control the $H^1$ norm of
$d$ then we control the $H^1$ norm of $u$.

For the following estimates and results, we define the
$H^p$-norm, $p \geq 1$,
of a function $u$ by 
\[
\| u \|_{H^p} := \left\| \partial_x^p u \right\|_{L^2},
\]
which is equivalent to the standard $H^p$-norm as we only consider functions
with vanishing mean, i.e.
$\int_0^{2\pi}u(x,t)\;\mathrm{d}x = 0$.
Note that in this setting Agmon's inequality
\begin{equation}\label{Agmon}
\|u\|_{L^\infty}\le\|u\|_{L^2}^{1/2}\|u_x\|_{L^2}^{1/2}
\end{equation}
holds with optimal constant $1$, which follows using the Fourier expansion.

A very important property of the 
surface growth equation (\ref{eqn:sfg}) is the existence of 
local solutions, which are smooth in space and time. 
This result is given by the following theorem from \cite{blomkerromito09} 
(Theorem 3.1).
\begin{theorem}\label{thm:regularity}
Let $u_0 \in H^1$, then there exists a time $\tau(u_0) > 0$ such that 
there is a unique solution $u \in C^0([0,\tau(u_0)),H^1)$ 
satisfying
\begin{enumerate}
	\item[1)] if $\tau(u_0) < \infty$, then 
		$\limsup\limits_{t\rightarrow\infty} \|u(t)\|_{H^1} = \infty$.
	\item[2)] $u$ is $C^\infty$ in both, space and time, for all 
		$(t,x) \in (0,\tau(u_0)) \times [0,2\pi]$.
\end{enumerate}
\end{theorem} 
Note that the theorem implies that lack of blowup in $H^1$ is sufficient 
to ensure that the solution exists for all time and is smooth. 
In particular, all of the manipulations we make in what follows are 
valid until the blowup time.

Throughout the rest of the paper we consider the solutions with 
initial data in $H^1$ whose existence is guaranteed by 
Theorem \ref{thm:regularity}, and approximations 
$\varphi \in H^4_{\text{per}}$ in space and $H^1$ in time.

\subsection{Energy estimate}
\label{sec:Energyestimate}
In this section we prove the key	estimate 
(\ref{eqn:upperBound}) on which the theorems of the following sections are 
based.

If we use the surface growth equation (\ref{eqn:sfg}) to find 
the evolution of  $d(x,t)$ and defining the residual of 
the approximation $\varphi$ by
\[ \RES := \varphi_t + \varphi_{xxxx}
 + ({\varphi_x}^2)_{xx} ,
\]
then we have
\[
d_t = -  d_{xxxx} - ({u_x}^2)_{xx}
+ ({\varphi_x}^2)_{xx} - \RES.
\]
By replacing $u$ with $d+\varphi$ we obtain
\[
d_t
= - d_{xxxx} - ({d_x}^2)_{xx}
- 2 (d_x \varphi_x)_{xx}
- \RES.
\]
For the $H^1$-norm we have
\begin{align*}
\frac{1}{2} \partial_t \|d_x\|^2 = &
	\underbrace{\left\langle d_{xx} , d_{xxxx}
	\right\rangle}_{\mathrm{A}}
 + \underbrace{2\<d_{xx},\left(d_x\varphi_x\right)_{xx}\>
}_{\mathrm{B}}
 + \underbrace{\<d_{xx},({d_x}^2)_{xx}\>}_{\mathrm{C}}
 + \underbrace{\<d_{xx},\RES\> }_{\mathrm{D}},
\end{align*}
where $\langle \cdot , \cdot \rangle$ is the $L^2$ scalar product.
Now consider these terms separately. Integrating by parts 
we obtain $\mathrm{A}=- \left\|d_{xxx} \right\|_{L^2}^2$.
Secondly,
\begin{align*}
\mathrm{B} =  -2 \int_0^{2\pi} d_{xxx}(d_x\varphi_x)_x\,\mathrm{d}x
= \;\int_0^{2\pi} \left(d_{xx}\right)^2
	\varphi_{xx} \; \mathrm{d}x
 - 2 \int_0^{2\pi} d_{xxx}d_x\varphi_{xx}\; \mathrm{d}x \;
\end{align*}
and so
\begin{align*}
|\mathrm{B}| \leq & \; \left\|d_{xx}\right\|_{L^2}^2
	\left\|\varphi_{xx} \right\|_{L^{\infty}} 
	+ 2 \left\|d_{xxx} \right\|_{L^2} \left\|d_x\right\|_{L^2} 
	\left\|\varphi_{xx}\right\|_{L^\infty} \\
\leq & \; 3 \left\|d_{xxx} \right\|_{L^2} \left\|d_x\right\|_{L^2} 
	\left\|\varphi_{xx}\right\|_{L^\infty} \\
\leq & \; \frac{1}{4} \left\| d_{xxx}\right\|_{L^2}^2
	+ 9 \left\| d_x \right\|_{L^2}^2 \left\|
	\varphi_{xx} \right\|_{L^\infty}^2,
\end{align*}
using interpolation and Young's inequality.
For $C$ we have
\begin{align*}
\mathrm{C} =  - \int_0^{2\pi} ({d_x}^2)_xd_{xxx}
	\; \mathrm{d}x
=  -2 \int_0^{2\pi} d_xd_{xx}d_{xxx}
\; \mathrm{d}x,
\end{align*}
hence using Agmon's inequality (\ref{Agmon}), interpolation, 
and Young's inequality,
\begin{align*}
|\mathrm{C}| \leq & \; 2 \left\|d_x\right\|_{L^2}
	\left\|d_{xx} \right\|_{L^\infty}
	\left\|d_{xxx} \right\|_{L^2}\\
 \leq &\; 2  \left\|d_x \right\|_{L^2}
	\left\|d_{xx} \right\|_{L^2}^{\frac{1}{2}}
	\left\|d_{xxx} \right\|_{L^2}^{\frac{3}{2}}\\
\leq & \; 2  \left\|d_x \right\|_{L^2}^{\frac{5}{4}}
	\left\|d_{xxx} \right\|_{L^2}^{\frac{7}{4}}\\
\leq & \; \frac{1}{4} \left\|d_{xxx} \right\|_{L^2}^2
	+ K	\left\|d_x \right\|_{L^2}^{10},
\end{align*}
where $K=7^7/4$; and for the remaining term
\begin{align*}
|\mathrm{D}| \leq  \; \left\| \RES \right\|_{H^{-1}}
	\left\| d_{xxx} \right\|_{L^2}
\le \frac{1}{4} \left\| d_{xxx}
	\right\|_{L^2}^2 + \left\| \RES
	\right\|_{H^{-1}}^2.
\end{align*}
Combining these estimates 
and applying Poincaré inequality with the optimal constant 
$\omega = 1$, we obtain
\begin{align} \label{eqn:upperBound}
\frac{1}{2} \partial_t \|d\|_{H^1}^2 &\leq - \frac{1}{4}
	\left\| d\right\|_{H^3}^2 + K
	\left\|d_x \right\|_{L^2}^{10} +
	\left\| \RES \right\|_{H^{-1}}^2 + 9 \left\| d \right\|_{H^1}^2
	\left\| \varphi_{xx} \right\|_{L^\infty}^2 \nonumber \\
	& \leq K \left\|d \right\|_{H^1}^{10} + 
		\Big( 9 \left\| \varphi_{xx} \right\|_{L^\infty}^2 - \frac{1}{4} 
		\Big) \left\|d \right\|_{H^1}^2 +	\left\| \RES \right\|_{H^{-1}}^2
\end{align}
which is a scalar differential inequality of type
\begin{equation}
\dot\xi \leq b\xi^5 +\left(a(t)-c\right)\xi + f(t),
	\label{eqn:estType}
\end{equation}
and by standard ODE comparison principles
a solution of the equality in (\ref{eqn:estType})
provides an upper bound for $\|d \|_{H^1}^2$.

\subsection{Time and smallness conditions}

We need two important properties of the surface growth model, 
which we will prove now. These are for equations like
Navier–Stokes well known facts, namely:  
that smallness of the solution implies global 
uniqueness and that solutions are actually small after some time by 
energy-type estimates. These results go back to Leray (\cite{leray34}), 
more modern discussions can be found in \cite{constantinfoias88} (Theorem 9.3) 
and in a setting that parallels the treatment here in 
\cite{robinsonsadowski08}.
For our model similar results for the critical $H^{1/2}$-norm 
can be found in \cite{blomkerromito09}. But for our results,  
we need to derive the precise values of constants 
in the $H^{1}$-norm, which were not determined before.

First, if the $H^1$-norm of a solution $u$ is smaller than some
constant $\varepsilon_0$,
we have global regularity of $u$. 
\begin{theorem}[Smallness Condition]\label{thm:smallness}
If for some $t \in [0,T]$ one has that $\| u(t) \|_{H^1}$
is finite on $[0,t]$ and
\[
\| u(t) \|_{H^1} < \frac{1}{2} =: \varepsilon_0,
\]
then we have global regularity (and thus uniqueness) of
the solution $u$ on $[0,\infty)$.
\end{theorem}
\begin{proof}
This is established by almost
the same estimates derived for the parts (A) and (C) in Section
\ref{sec:Energyestimate} and Young's inequality with constant $\delta >0$.
To be more precise:
\begin{align*}
	\frac{1}{2} \partial_t \|u\|_{H^1}^2
	&= - \|u_{xxx} \|_{L^2}^2 + \int_0^{2\pi}
		u_{xx}({u_x}^2)_{xx}
		 \; \mathrm{d}x\\
	&\leq - \| u \|_{H^3}^2 + 2 \| u \|_{H^3}^\frac{7}{4}
		\| u \|_{H^1}^\frac{5}{4} \\
	&\leq - \| u \|_{H^3}^2 + 2 \cdot
		\Big( \delta \| u \|_{H^3}^2 +
		\left( \frac{8}{7} \delta \right)^{-7} \frac{1}{8}
		\| u \|_{H^1}^{10} \Big)\\
	&\leq - \| u \|_{H^3}^2 \Big( 1 - 2\delta -
		\left( \frac{8}{7} \delta \right)^{-7} \cdot \frac{1}{4}
		\| u \|_{H^1}^{8} \Big).
\end{align*}
If $1 - 2\delta - \left( \frac{8}{7} \delta \right)^{-7}
\cdot \frac{1}{4} \| u\|_{H^1}^{8} > 0$, then we obtain a global
bound on $\| u \|_{H^1}^2$. The optimal choice for the constant
from Young inequality is $\delta = \frac{7}{16}$ and with this 
value it follows, that if $\| u(t) \|_{H^1} < \frac{1}{2}$ we have a negative 
derivative and the norm decays over time and is therefore bounded.
\end{proof}
The second property is that, based on the smallness condition, we can
determine a time $T^*$, only depending on the initial value $u(0)$,
such that $\| u(T^*) \|_{H^1} < \varepsilon_0.$
\begin{theorem}[Time Condition]\label{thmtc}
If a solution u is regular up to time
\[
T^*(u(0)) := \frac{1}{\varepsilon_0^2} \| u(0) \|_{L^2}^2
= 4 \left\| u(0) \right\|_{L^2}^2,
\]
then we have global regularity of the solution $u$.
\end{theorem}

At the risk of labouring the point, we only need to verify regularity of
a solution starting at $u(0)$ up to time $T^*(u(0))$, and
from that point on regularity is automatic.

\begin{proof}
As an a-priori estimate we have
\[
\partial_t \| u \|_{L^2}^2
= - \left\|u_{xx} \right\|_{L^2}^2
\]
and thus
\[
\int_0^T \left\|u_x(s) \right\|_{L^2}^2 \; \mathrm{d}s
\leq \int_0^T \left\| u_{xx}(s) \right\|_{L^2}^2 \; \mathrm{d}s
\leq \left\| u(0) \right\|_{L^2}^2
\]
where we used the Poincaré inequality with constant $\omega=1$.
If we now assume that
$\left\| u_x(s) \right\|_{L^2} > \varepsilon_0$ for all
$s \in [0,T]$, then
\[
T \varepsilon_0^2 < \left\| u(0) \right\|_{L^2}^2
\quad \text{or}\quad
T < \frac{1}{\varepsilon_0^2} \left\| u(0) \right\|_{L^2}^2
\]
This means, that if we wait until time
$T^* := \frac{1}{\varepsilon_0^2} \left\| u(0) \right\|_{L^2}^2$,
we know that $\left\| u(t) \right\|_{H^1} \leq \varepsilon_0$
for at least one $t \in [0,T^*]$
and we have global regularity by the smallness condition,
if there was no blowup before time $T^*$. 
\end{proof}

\section{ODE estimates}
\label{sec:ode}

We present several methods to bound solutions of
ODEs of the type (\ref{eqn:upperBound}).
In this section we give the results for the scalar ODE,
and present applications in the next section.

Let us first state a lemma of Gronwall type, based on comparison 
principles for ODEs, for which we will only give the idea of a proof.
\begin{lemma}[Gronwall]
\label{lem:Gronwall}
Let $a,b \in L^1([0,T],\mathbb{R})$ and
$x \in W^{1,1}([0,T],\mathbb{R})\cap C^0([0,T],\mathbb{R})$ such that
\[
 \dot x \leq a(t)x + b(t) \qquad \forall t \in [0,T] .
\]
Then for all $t \in [0,T]$
\[
x(t) \leq \exp \Big( \int_0^t a(s) \; \mathrm{d}s \Big) x(0)
+ \int_0^t \exp \Big( \int_s^t a(r) \; \mathrm{d}r \Big)
 b(s) \; \mathrm{d}s
\;.
\]
\end{lemma}
\begin{proof}[Idea of Proof]
Consider the function 
\[
u(t)=x(t)\exp\{- \int_0^t a( s)ds\}
\quad \text{with} \quad  
u'( t) \leq b( t) \exp\{- \int_0^t a(s)ds\}.
\] 
Integrating and solving for $x$ yields the result.
\end{proof}
\begin{lemma}\label{lem:cp}
Consider two functions
$x,u \in W^{1,1}([0,T],\mathbb{R}^+_0) \cap C^0([0,T],\mathbb{R}^+_0)$ 
such that
\[ \dot x \leq c(t)x^p + e(t) \qquad x(0) = x_0
\]
with $p > 1$, $c \in L^1([0,T],\mathbb{R}^+_0)$ and
$e \in L^1([0,T],\mathbb{R}^+_0)$, and let $u$ be the solution of
\[
\dot u = c(t)u^p \qquad u(0) = x_0 + \int_0^T e(s)
\; \mathrm{d}s .
\]
Then $x(t) \leq u(t)$ for all $t \in [0,T]$.
\end{lemma}

\begin{proof}
First note that if $e \equiv 0 $ on $[0,T]$ then by using the
standard comparison principle it follows that
$u(t) \geq x(t)$ for all $t \in [0,T]$.

So now we assume that $\int_0^T e(s) \; \mathrm{d}s > 0$. For a
contradiction, suppose that there exists a time $t^* \in [0,T]$ such
that $t^*:= \inf \left\{t > 0:x(t) = u(t) \right\}$.
Because of the continuity of $u(t)$ and $x(t)$, and $u(0) > x(0)$ due
to our initial assumption $\int_0^T e(s) \; \mathrm{d}s > 0$,
it follows that $t^* > 0$. From the definition $u(t) > x(t)$ for
all $t \in [0,t^*)$, and thus
\begin{align*}
0 = u(t^*) - x(t^*) & \geq u(0) - x(0) 
	- \int_0^{t^*} e(s) \; \mathrm{d}s + 
	\int_0^{t^*} c(s)\left( u(s)^p - x(s)^p \right) \; \mathrm{d}s\\
& = \int_{t^*}^T e(s) \; \mathrm{d}s
	+ \int_0^{t^*} c(s)( u(s)^p - x(s)^p)
	\; \mathrm{d}s,
\end{align*}
which is strictly positive provided that
$\int_0^{t^*} c(s) \; \mathrm{d}s >0$.

If $\int_0^{t^*} c(s) \; \mathrm{d}s = 0$, then as $c\ge0$ we obtain
\[
x(t) \leq x(0) + \int_0^{t} e(s) \; \mathrm{d}s
\leq x(0) + \int_0^{T} e(s) \; \mathrm{d}s = u(t)
\quad \forall t \in [0,t^*],
\]
and we can repeat the above argument on the interval
$[t^*,T]$ to obtain a contradiction.\end{proof}

\begin{theorem}[CP-Type I] \label{thm:cpType}
Assume $x \in W^{1,1}([0,T],\mathbb{R}^+_0)
\cap C^0([0,T],\mathbb{R}^+_0)$ such that
\[ \dot x \leq c(t)x^p + e(t), \qquad x(0) = x_0 \]
with $p > 1$, $c \in L^1([0,T],\mathbb{R}^+_0)$ and
$e \in L^1([0,T],\mathbb{R}^+_0)$. Then for all $t\in [0,T]$, 
as long as the right-hand side is finite,
\[
x(t) \leq \Big( x_0 + \int_0^t e(s) \; \mathrm{d}s \Big)
\Big\{ 1 -(p-1)\Big[ x_0 + \int_0^t e(s) \;
\mathrm{d}s \Big]^{p-1} \!\!\!\int_0^t c(s) \;
\mathrm{d}s \Big\}^{-\frac{1}{p-1}}\;.
\]
\end{theorem}
\begin{proof}
Given the setting of Lemma \ref{lem:cp}, we can solve for $u(t)$. 
As $\mathrm{d}u = c(t) u^p \; \mathrm{d}t$,
a straightforward calculation shows that
\[
 u(t) = u(0)\Big( 1 -(p-1) u(0)^{p-1}
 \int_0^t c(s) \; \mathrm{d}s \Big)^{-\frac{1}{p-1}}
\]
as long as the right-hand side is finite.
Thus for all $t \in [0,T]$, as long as the right-hand side is finite,
\begin{align*}
x(t) &\leq \Big( x_0 + \int_0^T e(s) \; \mathrm{d}s \Big) \\
 & \; \times \Big\{ 1 -(p-1)\Big[ x_0 + \int_0^T e(s) \; \mathrm{d}s
 \Big]^{p-1} \!\!\!\int_0^t c(s) \; \mathrm{d}s \Big\}^{-\frac{1}{p-1}}
\end{align*}
This holds particularly when $T=t$.
\end{proof}

We now extend this result to differential inequalities of the form
\[ \dot x \leq b(t)x^p + a(t)x + f(t), \]
where $p>1$, $f,b \in L^1([0,T],\mathbb{R}^+_0)$ and
$a \in L^1([0,T],\mathbb{R})$, as our inequality 
(\ref{eqn:upperBound}) is of this type. 
\begin{corollary}[CP-Type II] \label{cor:cpType}
Assume $x \in W^{1,1}([0,T],\mathbb{R}^+_0) \cap C^0([0,T],\mathbb{R}^+_0)$
such that
\[ \dot x \leq b(t)x^p + a(t)x + f(t), \]
with $p>1$, $b,f \in L^1([0,T],\mathbb{R}^+_0)$ and
$a \in L^1([0,T],\mathbb{R})$. Then for all $t\in [0,T]$, as 
long as the right-hand side is finite,
\begin{align*}
x(t) \leq \; \e^{A(t)} & \Big( x_0 + \int_0^t
	\tilde{f}(s) \; \mathrm{d}s \Big) \\
&\times\Big\{ 1 - (p-1)\cdot \Big[ x_0 + \int_0^t \tilde{f}(s)
	\; \mathrm{d}s \Big]^{p-1} \int_0^t \tilde{b}(s) \;
	\mathrm{d}s \Big\}^{-\frac{1}{p-1}}
\end{align*}
where
\[
 \tilde{b}(t) = \; b(t)  \e^{(p-1) A(t)}, \quad
 \tilde{f}(t) = \; \e^{-A(t)}  f(t) ,
 \quad\text{and}\quad A(t) = \int_0^t a(s) \; \mathrm{d}s \;.
\]
\end{corollary}
\begin{proof}
Consider the substitution $y(t) = \e^{- A(t)} x(t)$ with
$ A(t) = \int_0^t a(s) \; \mathrm{d}s $. It follows that
\begin{align*}
\dot y &= - a(t)y + \e^{- A(t)}\dot x \\
& \leq - a(t)y + \e^{- A(t)} \left( b(t)x^p + a(t)x + f(t) \right) \\
&= \underbrace{b(t) \e^{(p-1) A(t)}}_{\tilde{b}(t)}\,
	 y^p + \underbrace{\e^{-A(t)}  f(t)}_{\tilde{f}(t)}
\end{align*}
with $\tilde{b}(t) \geq 0$ and $\tilde{f}(t) \geq 0$ for all $t \in [0,T]$.
Here we can apply Theorem \ref{thm:cpType} and obtain
\[
 y(t) \leq \Big( y_0 + \int_0^t \tilde{f} \;\mathrm{d}s \Big)
 \Big\{ 1 -(p-1)\Big[ y_0 + \int_0^t \tilde{f} \;
\mathrm{d}s \Big]^{p-1} \!\!\! \int_0^t \tilde{b} \;
\mathrm{d}s \Big\}^{-\frac{1}{p-1}}.
\]
Now substitute back with $x(t) = \e^{A(t)} y(t)$.
\end{proof}

\section{Verification methods}
\label{sec:res}

We now outline three techniques for numerical verification. 
All of them are based on the key estimate (\ref{eqn:upperBound}) 
for the difference $d$ between a smooth approximation 
$\varphi$ and a smooth local solution. The first is additionally
based on the simple Gronwall Lemma \ref{lem:Gronwall}, the second on
Corollary \ref{cor:cpType}, and the third is similar to 
the second method, but restarts the
estimation after a series of short time-steps.

\subsection{First method}

This is based directly on the simple Gronwall Lemma \ref{lem:Gronwall}. 
Assuming a poor bound to control the nonlinearity, 
we prove a better error estimate.
\begin{theorem}\label{thm:method1}
Let $K^*=(8K)^{-1/8}=(2\times 7^7)^{-1/8}$.	As long as
\begin{equation}\label{cond1}
\| d(0)\|_{H^1}^2 \e^{A(t)}
+ 2 \int_0^t \| \RES(s) \|_{H^{-1}}^2  \e^{ ( A(t) - A(s) )}
\; \mathrm{d}s \leq K^*,
\end{equation}
we have	
\[
\|d(t)\|_{H^1}^2 \leq\| d(0)\|_{H^1}^2  \e^{A(t)}
+ 2 \int_0^t \| \RES(s) \|_{H^{-1}}^2
 \e^{ ( A(t) - A(s) )} \; \mathrm{d}s,
\]
where $A(t) = - \frac{1}{4}t + 18 \int_0^t
\left\| \varphi_{xx}(\tau) \right\|_{L^\infty}^2
\mathrm{d}\tau.$
\end{theorem}

Note that the condition in (\ref{cond1}) involves only the numerical
solution $\varphi$.

\begin{proof}
It follows from the inequality (\ref{eqn:upperBound})
that as long as $\left\|d\right\|_{H^1}^8 \leq (8K)^{-1}$
we obtain
\begin{align*}
\partial_t \|d\|_{H^1}^2 \leq & - \frac{1}{4} \| d
	\|_{H^1}^2 + 2 \| \RES \|_{H^{-1}}^2  + 18 \left\| d \right\|_{H^1}^2
	\left\| \varphi_{xx} \right\|_{L^\infty}^2.
\end{align*}
Now we can apply Lemma \ref{lem:Gronwall} to deduce that
\begin{align*}
\|d(t)\|_{H^1}^2 \leq & \;\|d(0)\|_{H^1}^2 \exp
	\Big\{ - \frac{t}{4} + 18 \int_0^t \|
	\varphi_{xx}(\tau)\|_{L^\infty}^2 \mathrm{d}\tau \Big\}  \\
& + 2 \int_0^t \| \RES(s) \|_{H^{-1}}^2 \exp
	\Big\{ - \frac{t-s}{4} + 18 \int_s^t \| \varphi_{xx}(\tau)
	\|_{L^\infty}^2 \; \mathrm{d}\tau \Big\} \mathrm{d}s.
\end{align*}
\end{proof}
Please note that if the bound exceeds $K^*$, 
Theorem \ref{thm:method1} makes no assertions on $\|d(t)\|_{H^1}^2$.

\subsection{Second method}

Here we present the more sophisticated method based on direct 
application of Corollary \ref{cor:cpType} (CP-Type II).
\begin{theorem}\label{thm:method2}
As long as the right-hand side is finite, the following 
inequality holds for $d(t)$:
\begin{align*}
\|d(t)\|_{H^1}^2 &\leq \e^{A(t)}\Big( \|d(0)\|_{H^1}^2
	+ \int_0^t \tilde{f}(s) \; \mathrm{d}s \Big) \\
& \qquad \times\Big\{ 1 - 4
	\Big[ \|d(0)\|_{H^1}^2 + \int_0^t \tilde{f}(s) \; \mathrm{d}s \Big]^4
	\int_0^t \tilde{b}(s) \;\mathrm{d}s \Big\}^{-1/4}
\end{align*}
with
\[
\tilde{b}(t) = K\e^{4  A(t)},\qquad
\tilde{f}(t) = \e^{-A(t)} \left\| \RES(t) \right\|_{H^{-1}}^2
\]
and
\[
 A(t)
 = -\frac{t}{4} + \int_0^t 9
 \left\|  \varphi_{xx}(s) \right\|_{L^\infty}^2 \; \mathrm{d}s.
\]
\end{theorem}

Again, the condition for regularity provided by the theorem depends
only on the numerical solution $\varphi$.

\begin{proof}
Apply Corollary \ref{cor:cpType} (CP-Type II) to our
inequality (\ref{eqn:upperBound}).
The corresponding functions are
\[
b(t) = \frac{7^7}{4}, \quad a(t)= 9
\left\| \varphi_{xx}(t) \right\|_{L^\infty}^2
- \frac{1}{4}, \quad f(t)= \left\| \RES(t) \right\|_{H^{-1}}^2,
\]
which immediately give us the statement of the theorem.
\end{proof}

\subsection{Second method with restarting}
The previous method can be further improved by introducing something
that can be best described as ``restarting". Instead of estimating
over the whole time interval $[0,T]$ at once, we estimate to
some smaller $t^*$ and use the resulting upper bound as the
new initial value.
\begin{theorem}\label{thm:method2R}
Given any arbitrary partition $\{t_i\}_{0 \leq i \leq n}$ of the
interval $[0,T]$ with $t_0 = 0$ and $t_n = T$, then by
Theorem \ref{thm:method2} we have for all $1 \leq i \leq n$
\begin{align*}
z(0) &:= \|d(0)\|_{H^1}^2 \\
\|d(t_i)\|_{H^1}^2 & \leq \e^{A(t_i)}  \Big( z(t_{i-1}) +
	\int_{t_{i-1}}^{t_i} \tilde{f}(s) \; \mathrm{d}s \Big) \\
& \; \times \Big\{ 1 - 4  \Big[ z(t_{i-1}) +
	\int_{t_{i-1}}^{t_i} \tilde{f}(s) \; \mathrm{d}s \Big]^4
	 \int_{t_{i-1}}^{t_i} \tilde{b}(s) \;\mathrm{d}s \Big\}^{-1/4}\\
& =: z(t_i)
\end{align*}
as long as the right-hand side is finite, where for $t \in (t_{i-1},t_i]$
\[ \tilde{b}(t) = K \e^{4  A(t)}, \qquad
 \tilde{f}(t) = \e^{-A(t)}  \left\| \RES(t) \right\|_{H^{-1}}^2
\]
and
\[
 A(t) = -\frac{1}{4}  (t - t_{i-1}) +
 \int_{t_{i-1}}^{t} 9  \|  \varphi_{xx}(s) \|_{L^\infty}^2
 \; \mathrm{d}s.
\]
\end{theorem}
\begin{proof}
Given some arbitrary partition $\{t_i\}_{0 \leq i \leq n}$ of the
interval $[0,T]$ with $t_0 = 0$ and $t_n = T$, we define our
new method as follows.

First, we apply Theorem \ref{thm:method2} to the interval $[0,t_1]$
\begin{align*}
z(0) &:= \|d(0)\|_{H^1}^2 \\
\|d(t_1)\|_{H^1}^2 & \leq \e^{A(t_1)}
	\Big( z(0) + \int_0^{t_1} \tilde{f}(s) \; \mathrm{d}s \Big)  \\
& \; \times \Big\{ 1 - 4  \Big[ z(0) +
	\int_0^{t_1} \tilde{f}(s) \; \mathrm{d}s \Big]^4
	\int_0^{t_1} \tilde{b}(s) \;\mathrm{d}s \Big\}^{-1/4} \\
&=: z(t_1)\\
\intertext{and define the upper bound for
	$\|d(t_1)\|_{H^1}^2$ as $z(t_1)$. In the next step, $z(t_1)$ is
	taken as the new ``initial value" when we apply
	Theorem \ref{thm:method2} to the interval $[t_1,t_2]$.}
\|d(t_2)\|_{H^1}^2 & \leq \e^{A(t_2)}
	\Big( z(t_1) + \int_{t_1}^{t_2} \tilde{f}(s)
	\; \mathrm{d}s \Big) \\
& \; \times \Big\{( 1 - 4  \Big[ z(t_1) +
	\int_{t_1}^{t_2} \tilde{f}(s) \; \mathrm{d}s \Big]^4
	 \int_{t_1}^{t_2} \tilde{b}(s)
	\;\mathrm{d}s \Big\}^{-1/4} \\
\intertext{where $\tilde{b}(t),\tilde{f}(t)$ are defined as before,
	only $A(t)$ for $t \in (t_{i-1},t_i]$ changes to}
A(t) &= -\frac{1}{4}  (t - t_{i-1}) + \int_{t_{i-1}}^{t} 9
	 \left\| \varphi_{xx}(s) \right\|_{L^\infty}^2
	\; \mathrm{d}s.
\end{align*}
This procedure is now repeated for every interval of the partition.
\end{proof}

Let us give an informal argument that this method converges to a
solution of the ODE as $h\to0$.
Let $z(t)$ be a smooth interpolation of the discrete
points $z(t_i)$, $i=1,2,\ldots$ 
and $h=t_{j+1}-t_j$.
Then
\[
\partial_t z(t_j) \approx \frac{z(t_{j+1})-z(t_j)}h
\]
Using, $\int_{t_j}^{t_{j+1}} g \; \mathrm{d}s \approx g(t_j)h $
and the abbreviations $z(t_j)=z_j$, $A_j = A(t_j)$ and 
$\RES_j = \| \RES(t_j)\|_{H^{-1}}^2$,
we obtain from Theorem \ref{thm:method2R}
\[
\partial_t z(t_j) \approx \frac1h \Big[ \frac{\e^{A_{j+1}}
	( z_j +h\tilde{f}_j)}{( 1 - 4[z_j +
	h\tilde{f}_j]^4 h\tilde{b}_j)^{1/4}} - z_j\Big].
\]
Using $\tilde{b}_j=\frac{7^7}{4}\e^{4A_j} = \frac{7^7}{4} $ and
$\tilde{f}_j=\e^{-A_j}\RES_j= \RES_j,$ as $A_j=0$ yields
\[
\begin{split}
\partial_t z(t_j) & \approx \frac1h
	\Big[ \frac{ \e^{A_{j+1}}( z_j +h\RES_j)}{( 1 - [ z_j + h \RES_j]^4
	\cdot h 7^7)^{1/4}} - z_j\Big]\\
& \approx \frac1h \Big[ \frac{  \e^{A_{j+1}}z_j +h\e^{A_{j+1}}\RES_j-
	z_j\sqrt[4]{ 1 - 7^7h[ z_j + h\RES_j]^4} }
	{\sqrt[4]{ 1 - 7^7h[ z_j + h\RES_j]^4}} \Big] \\
& \approx \frac1h \Big[\e^{A_{j+1}}z_j+
	h\e^{A_{j+1}}\RES_j - z_j\sqrt[4]{ 1 - 7^7h[ z_j + h\RES_j]^4}\Big]
\end{split}
\]
Now using $\sqrt[4]{1-x}\approx 1 - \frac14 x + O(x^2)$ and 
$A_{j+1}=O(h)$ leads to
\[
\begin{split}
\partial_t z(t_j) & \approx \e^{A_{j+1}} \RES_j + \frac1h(\e^{A_{j+1}}-1)z_j
	+ z_j \frac14 7^7[ z_j + h\RES_j]^4 \\
&\approx \RES_j+  A'(t_j) z_j + \frac14 7^7 z_j^5.
\end{split}
\]
Recall that $\RES_j= \| \RES(t_j)\|_{H^{-1}}^2$
and $A'(t_j)= - \frac14 + 9\| \varphi_{xx}(t_j) \|_{L^\infty}^2$,
and we recover that $z$ solves (\ref{eqn:upperBound})
with equality in the limit $h\to 0$.

\section{Numerical examples}
\label{sec:num}

To perform numerical verification rigorously,
upper bounds for the three methods
need to be calculated that include rounding errors
(e.g using interval arithmetic). However,
as our aim here is to illustrate the general behavior and
feasibility of the three methods, we neglected these rounding errors.

Although our methods allow $\varphi$ to be any arbitrary approximation, 
that satisfies the boundary conditions, it
should be a reasonable choice, i.e.\ close to an expected solution,
for the methods to be successful. 

For our simulations we calculate an approximate solution
using a spectral Galerkin scheme with $N$ Fourier modes in space
and a semi-implicit Euler
scheme with step-size $h$ in time, yielding the
values $\varphi(t)$ for $t=0,h,2h,...$.
To calculate the residual of $\varphi$, these values are interpolated
piecewise linearly in time.

There are two ways to show global regularity for $u(0)=u_0$ using 
the numerical methods of the previous section:
\begin{itemize}
	\item show that the solution exists until the time $T^*(u_0)$ 
	(from Theorem \ref{thmtc}), since the solution is regular after 
	this time; or		
	\item show that 
	$\|\varphi(t)\|_{H^1} + \| d(t) \|_{H^1}<\varepsilon_0$ 
	for some $t>0$, since then Theorem \ref{thm:smallness} 
	guarantees global regularity.
\end{itemize}
Note that the second criterion might be more strongly influenced by rounding
errors than the first one.

In all of our figures the maximum time
is always $T^*$, as defined by Theorem \ref{thmtc},
rounded to the first decimal digit $+ 0.1$,
which is enough to show global existence.

\begin{figure}[!ht]%
	\begin{center}
	\hspace*{\fill} %
	\subfloat[Method 1]{\includegraphics[width=0.3\textwidth]
	{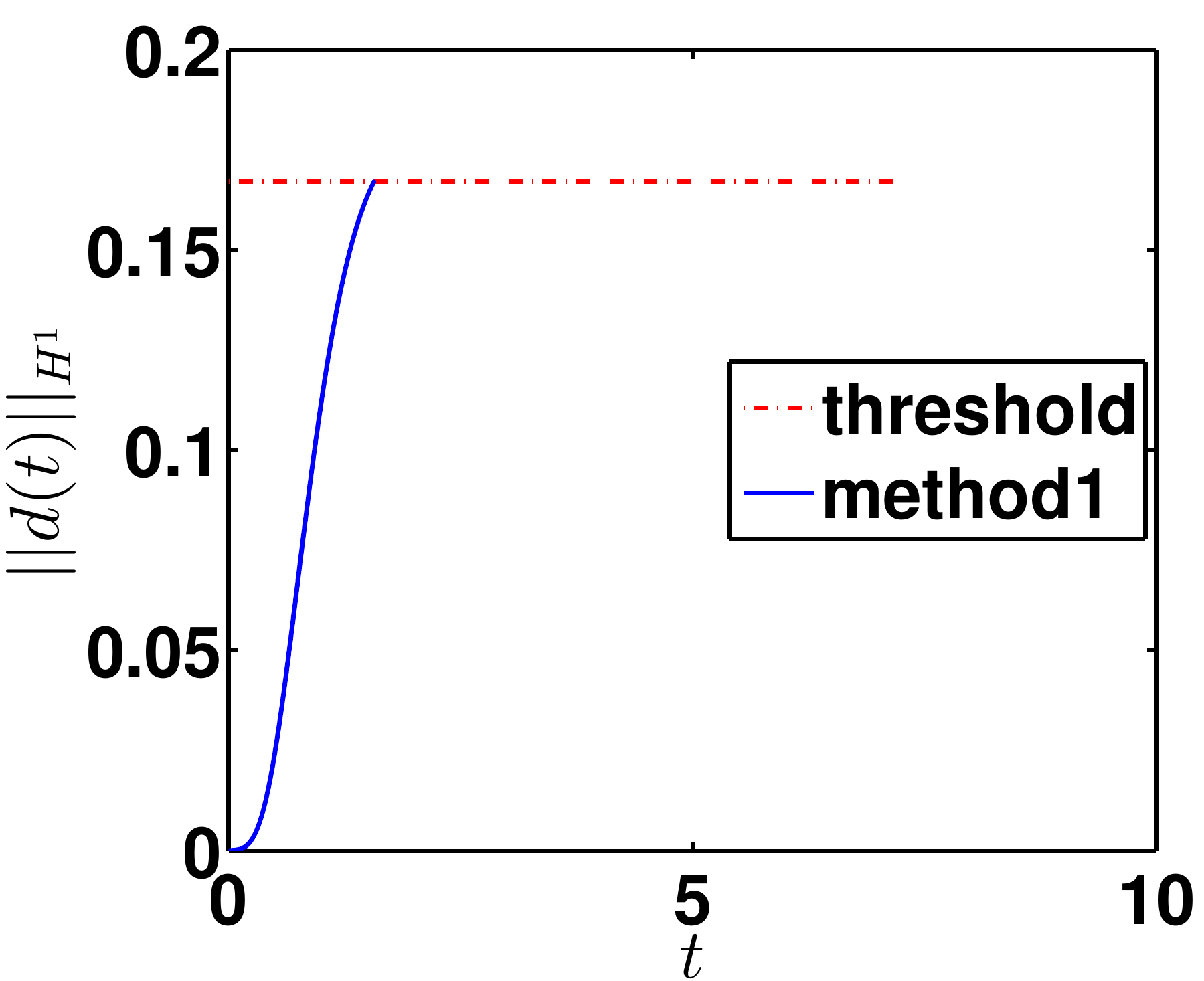}}
	\hspace*{\fill} %
	\subfloat[Method 2]{\includegraphics[width=0.3\textwidth]
	{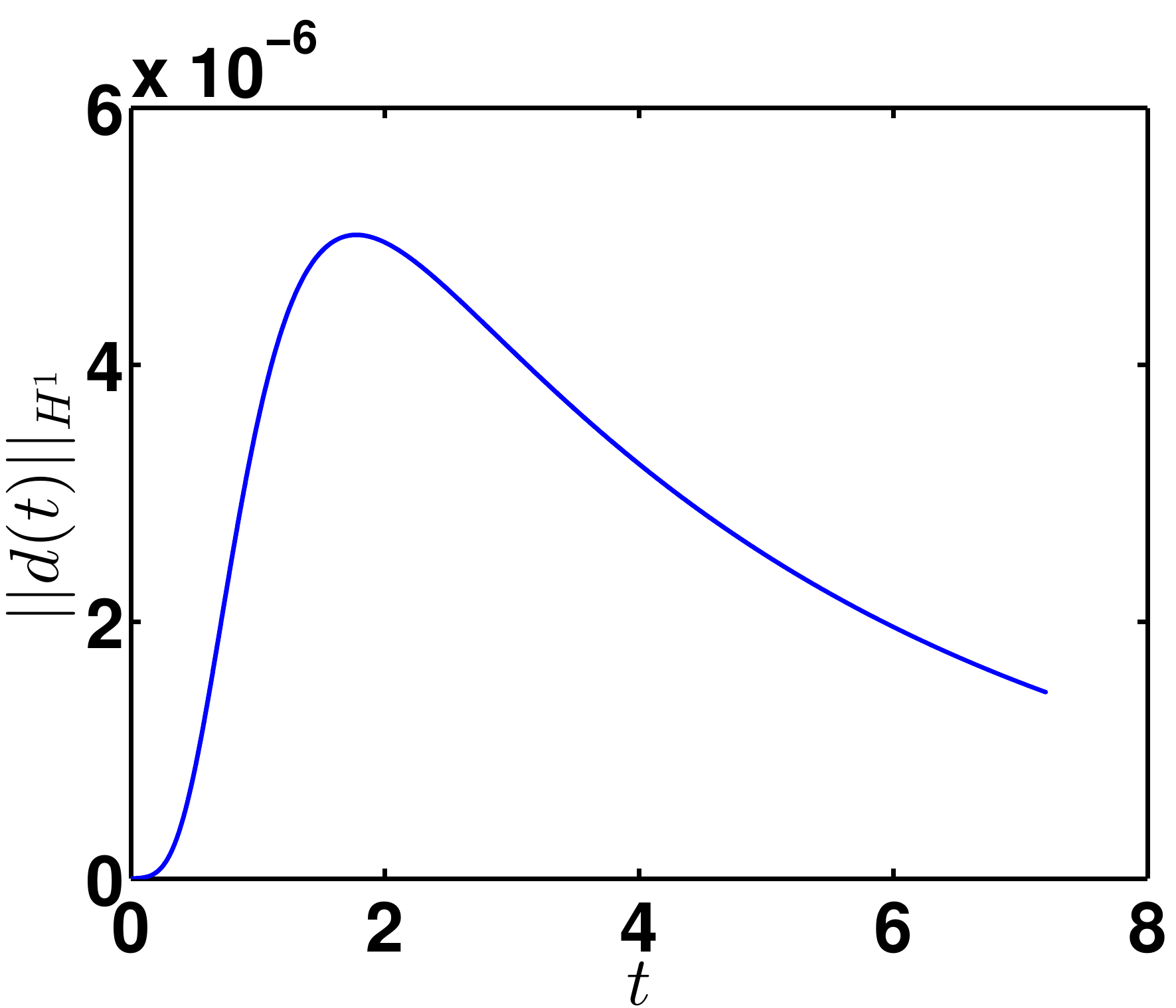}}
	\hspace*{\fill} %
	\subfloat[Method 3]{\includegraphics[width=0.3\textwidth]
	{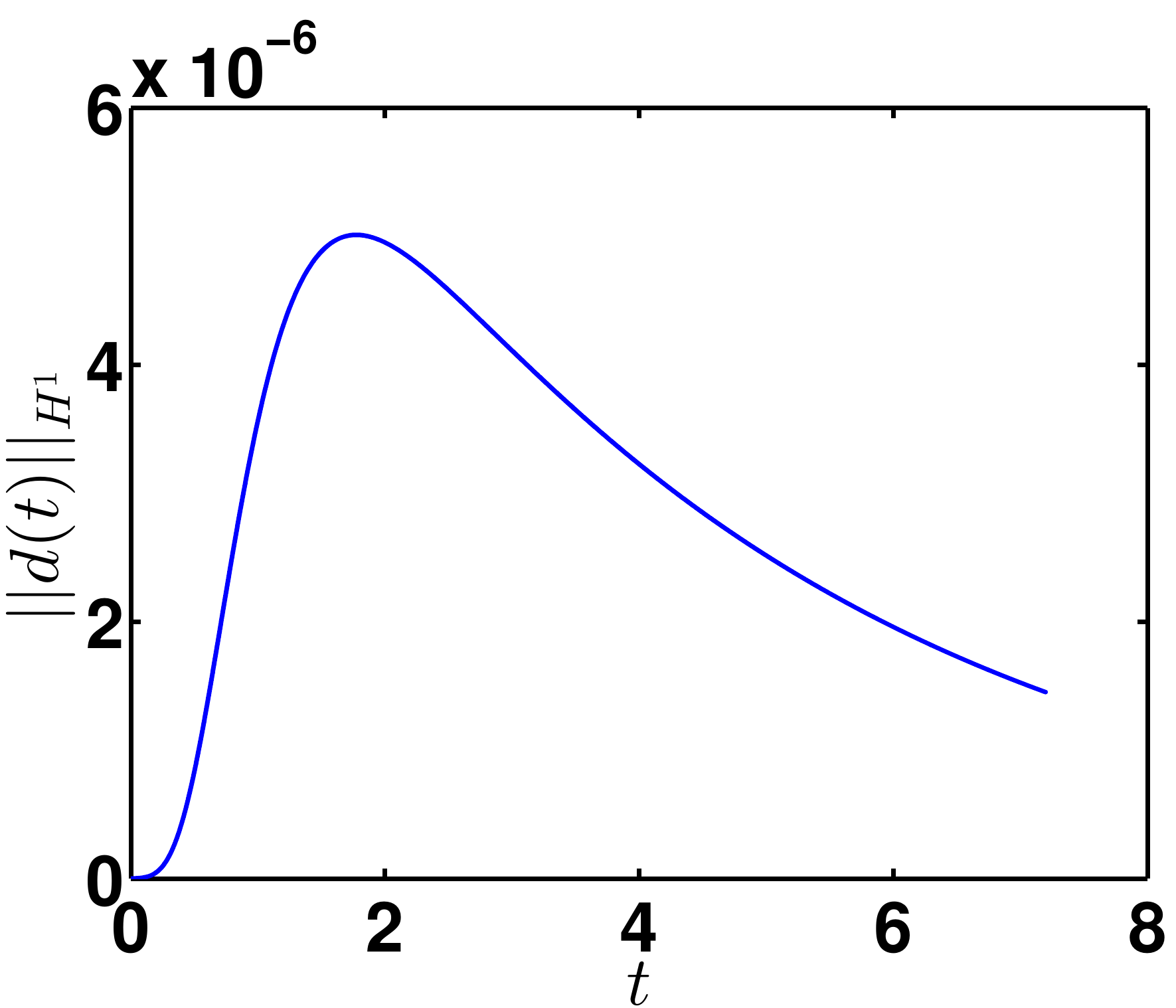}}
	\hspace*{\fill} %
	\\
	\hspace*{\fill} %
	\subfloat[Smallness Method 1]{\includegraphics[width=0.3\textwidth]
	{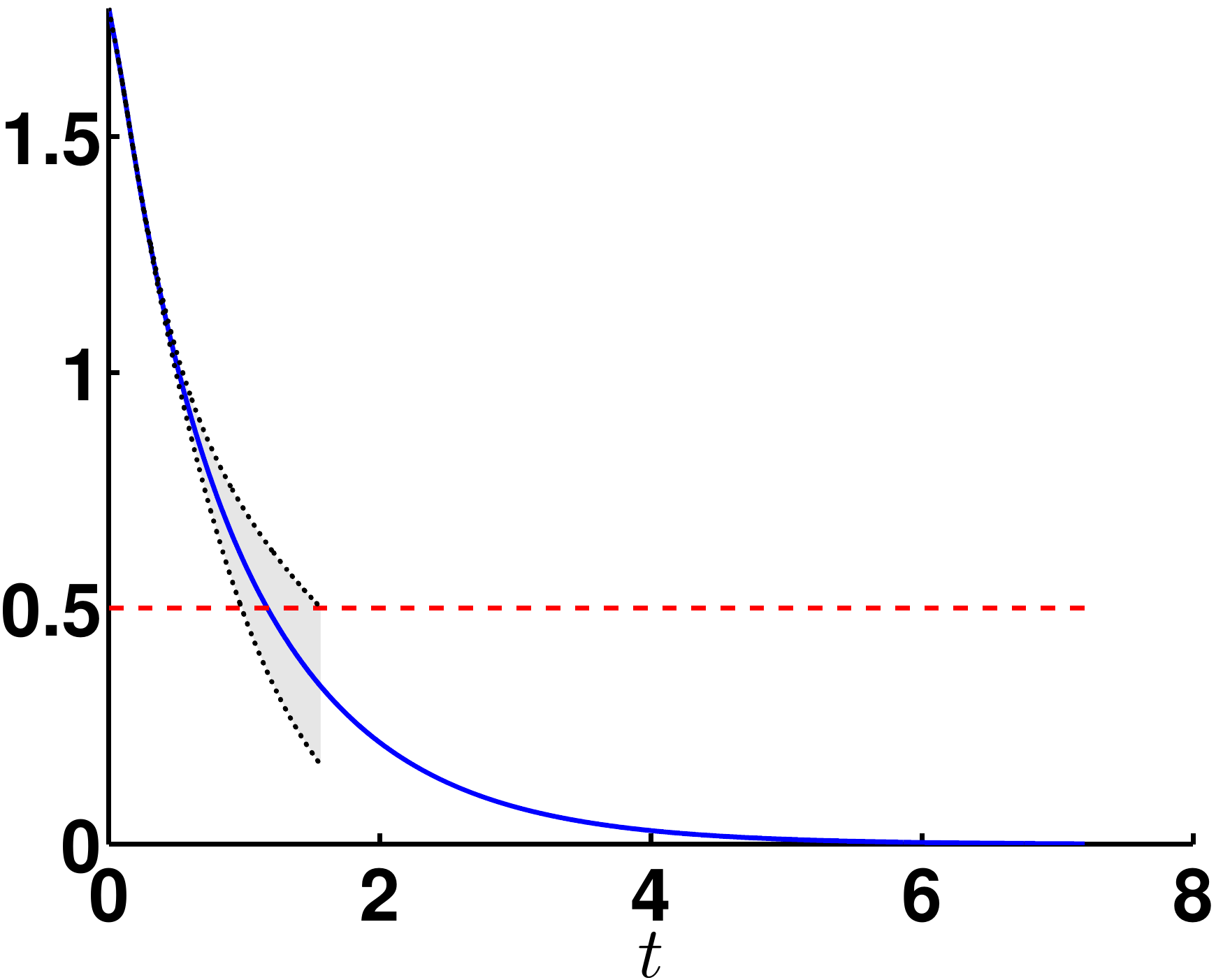} \label{img:ex1smallnessM1}}
	\hspace*{\fill} %
	\subfloat[Smallness Method 2]{\includegraphics[width=0.3\textwidth]
	{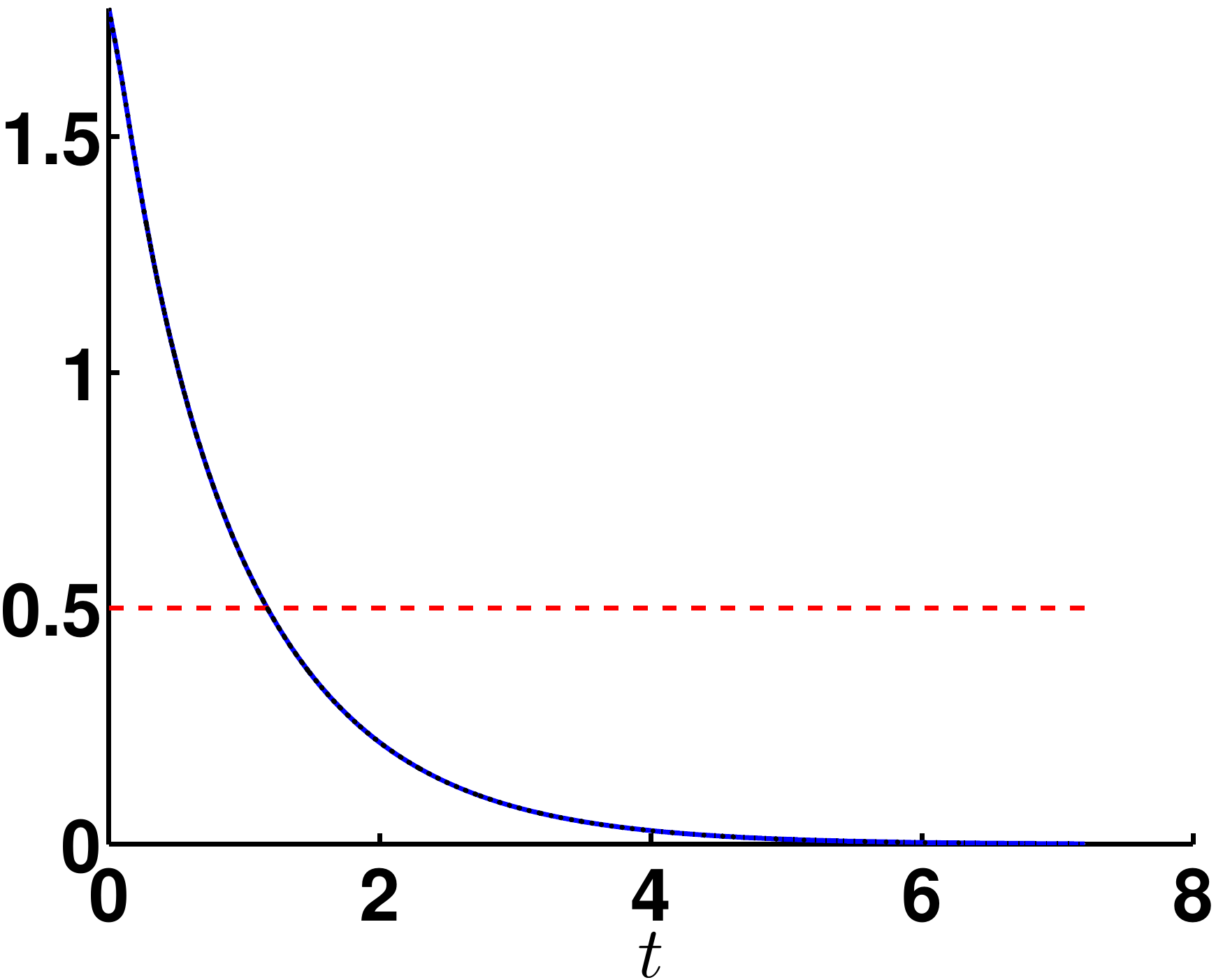} \label{img:ex1smallnessM2}}
	\hspace*{\fill} %
	\subfloat[Smallness Method 3]{\includegraphics[width=0.3\textwidth]
	{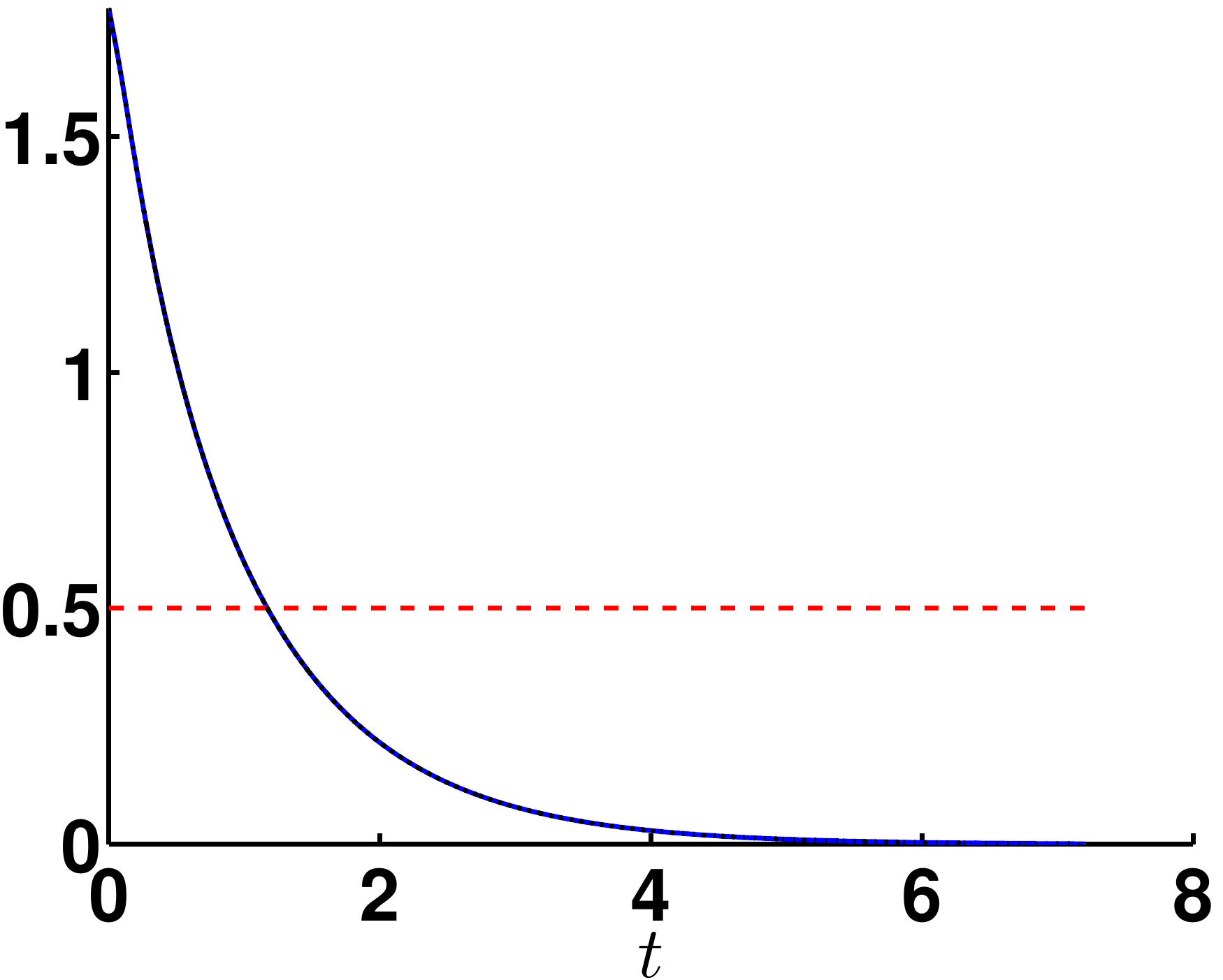} \label{img:ex1smallnessM3}}
	\hspace*{\fill} %
	\\
	\hspace*{\fill} %
	\subfloat[$\| \varphi_{xx} \|_{L^\infty}$]
	{\includegraphics[width=0.3\textwidth]{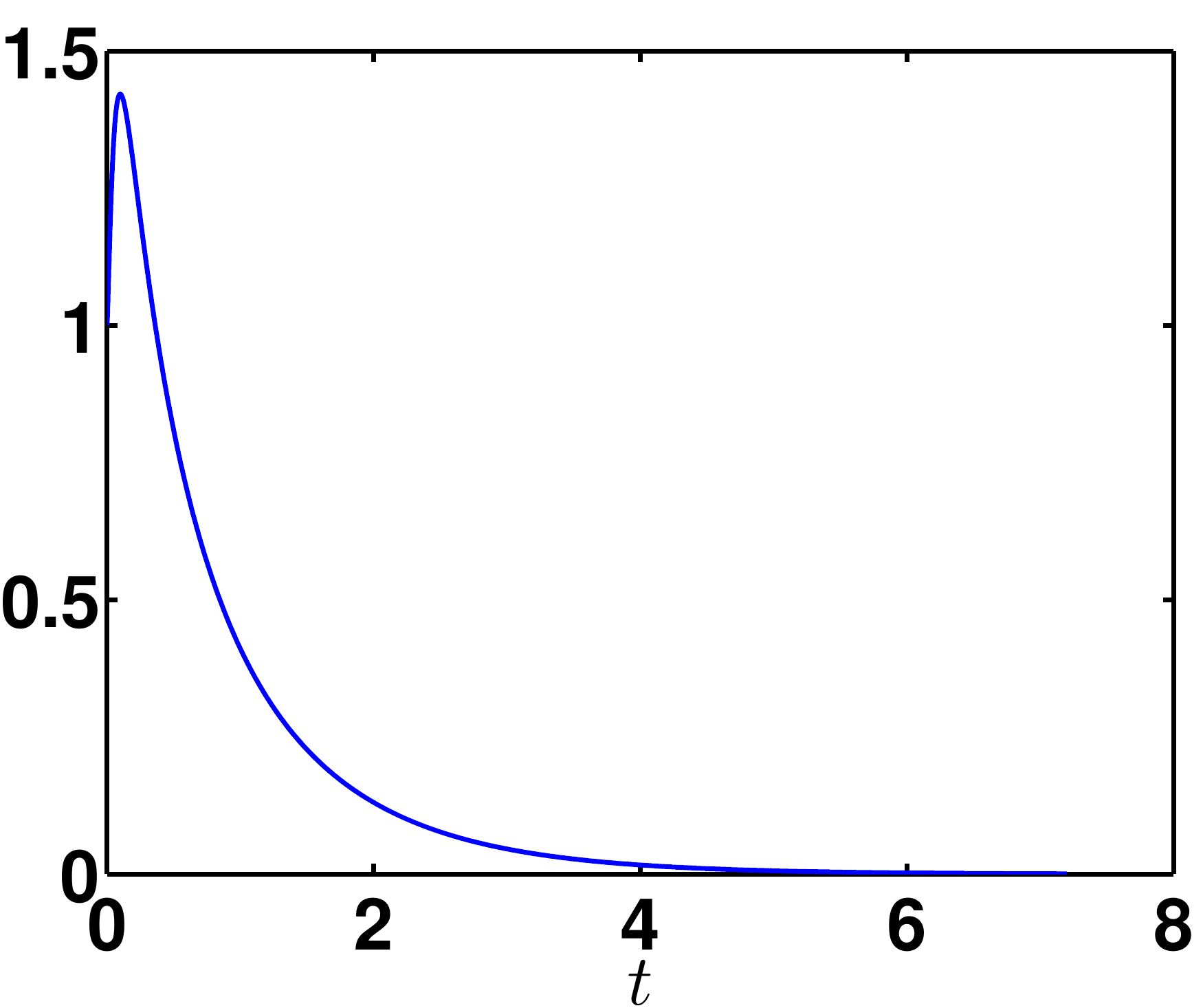}}
	\hspace*{\fill} %
	\subfloat[$\|\RES \|_{H^{-1}}$]{\includegraphics[width=0.3\textwidth]
	{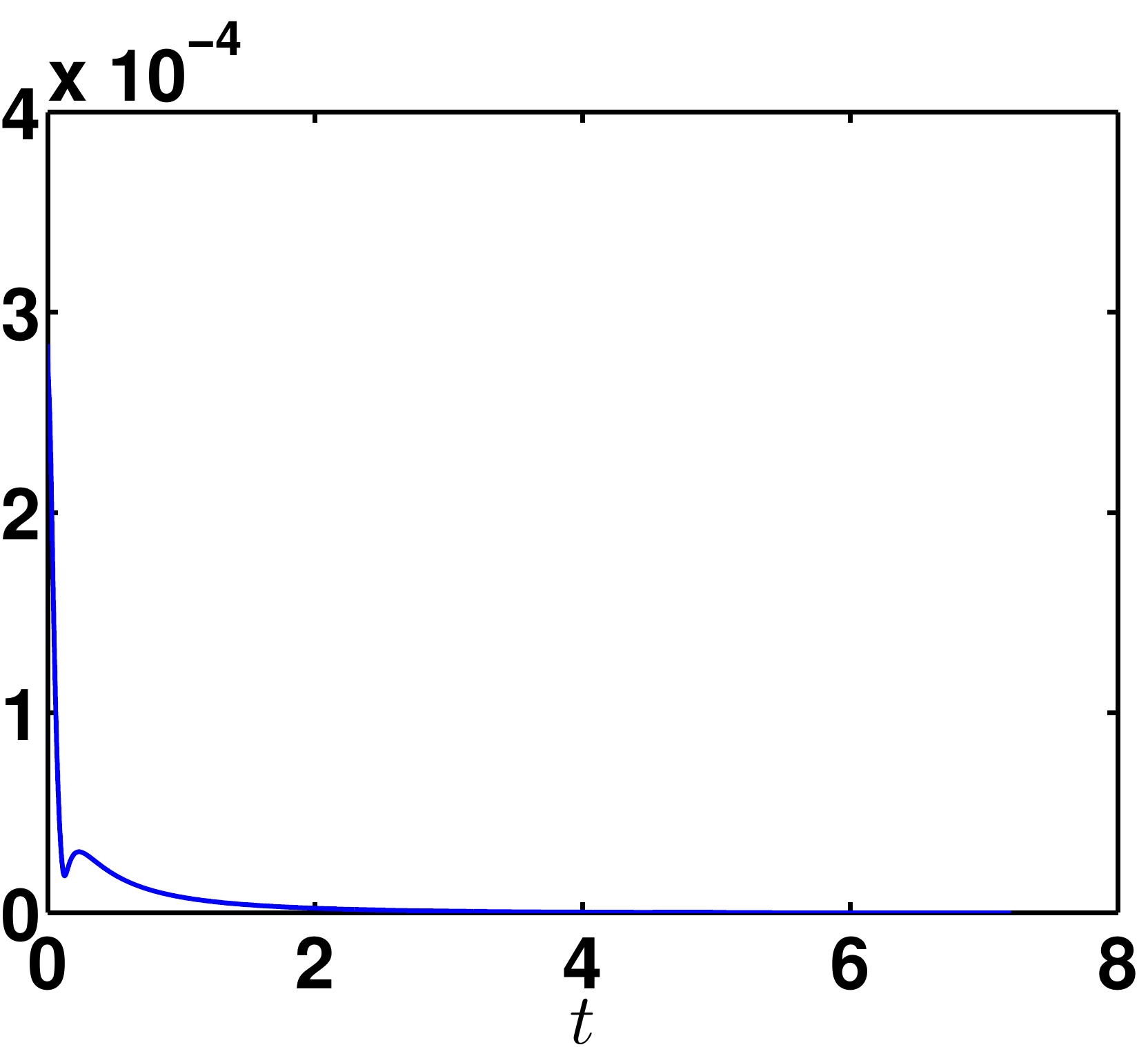}}
	\hspace*{\fill} %
	\subfloat[$\varphi$]{\includegraphics[width=0.3\textwidth]
	{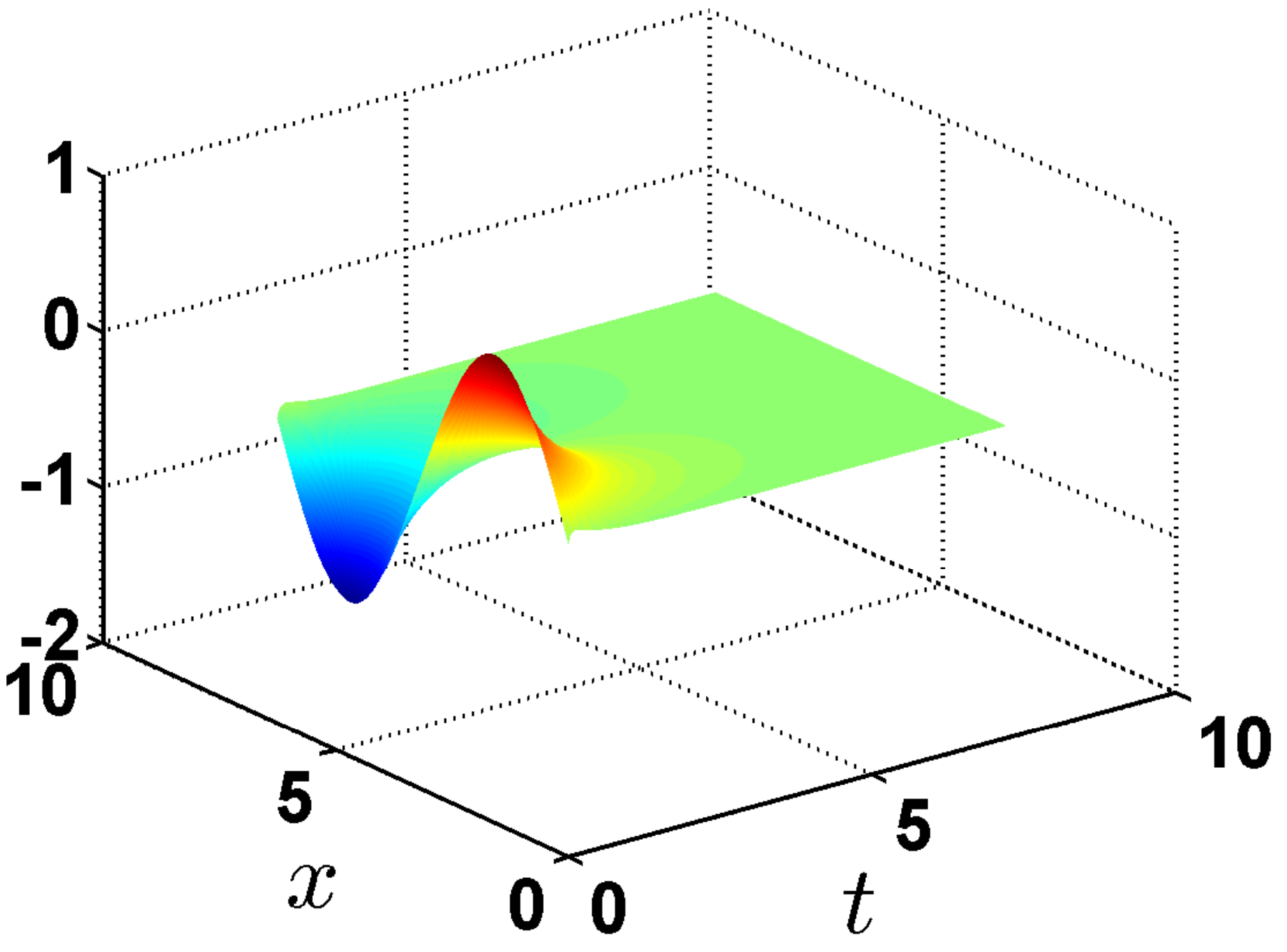}}
	\hspace*{\fill} %
	\\
	\caption{Initial value $u(x,0) = \sin(x)$, $N=128$ Fourier modes and
	 step-size $h=10^{-5}$. Methods 2 and 3 show global existence as
	 they stay bounded until time $T^*$ and even fulfill the
	 smallness criterion before time $T^*$. Method 1 fails as it hits
	 its threshold at approximately $t=1.5 < T^*$ and also before the
	 smallness criterion is reached.}
	 \label{img:example1}
	\end{center}
\end{figure}
In Figure \ref{img:example1} we have an initial value of $u(x,0) = \sin(x)$,
$N=128$ Fourier modes and a step-size of $h=10^{-5}$.
As Methods 2 and 3 stay bounded up to time $T^*$, we have global existence.
Method 1 fails because it hits its threshold at
approximately $t=1.5$, which is smaller than $T^*$. In the ``Smallness plots"
(Figures \ref{img:ex1smallnessM1}, \ref{img:ex1smallnessM2} and
\ref{img:ex1smallnessM3}) the grey area is the area
around $\|\varphi(t)\|_{H^1}$ (the solid line) with distance
$\|d(t)\|_{H^1}$, which is calculated by the respective method.
The dashed line is the critical value $\varepsilon_0 = \frac{1}{2}$.
As the order of $\|d(t)\|_{H^1}$ is $10^{-6}$ for Methods 2 and 3,
this area is not really visible in Figures \ref{img:ex1smallnessM2} and
\ref{img:ex1smallnessM3}.
An interesting detail
is, that in Figure \ref{img:ex1smallnessM1} it seems that the upper
bound reaches $\varepsilon_0$, but in fact it is still a tiny bit above
when Method 1 hits the threshold. To sum up, we have global regularity
for this initial value from Methods 2 and 3 by smallness and time criteria,
whereas both criteria fail for Method 1.

\begin{figure}[!ht]%
	\begin{center}
	\hspace*{\fill} %
	\subfloat[Method 1]{\includegraphics[width=0.3\textwidth]
	{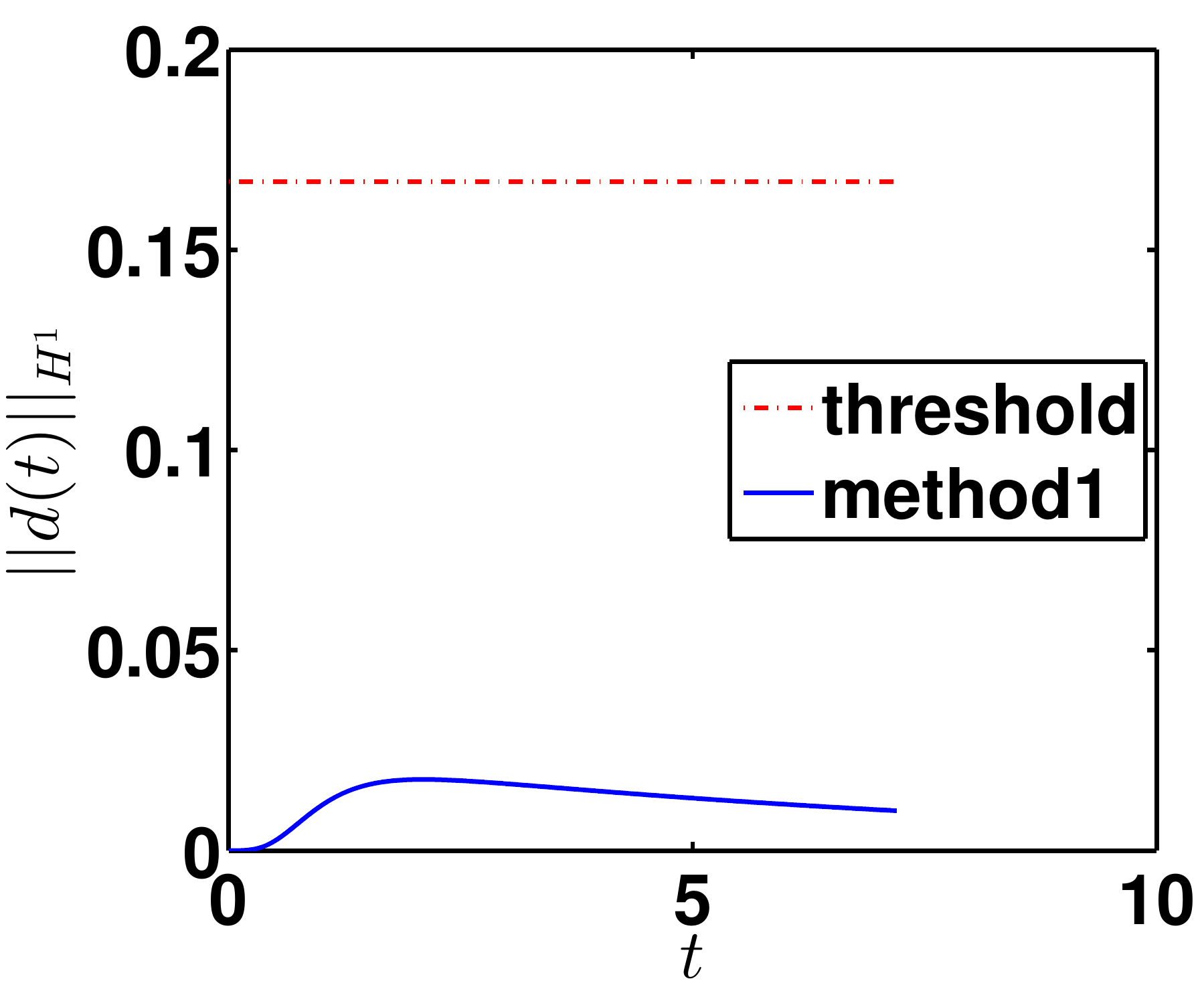}}
	\hspace*{\fill} %
	\subfloat[Method 2]{\includegraphics[width=0.3\textwidth]
	{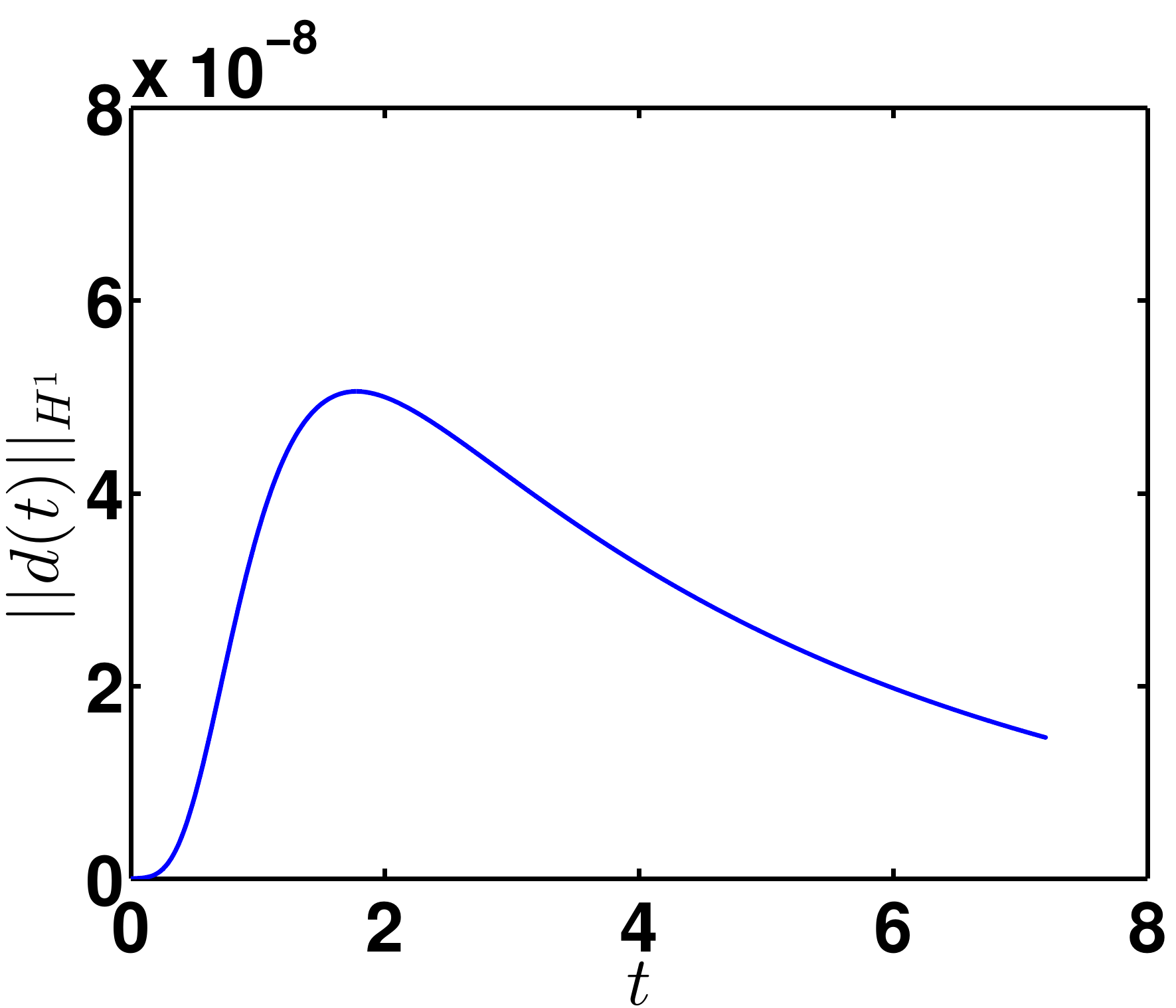}}
	\hspace*{\fill} %
	\subfloat[Method 3]{\includegraphics[width=0.3\textwidth]
	{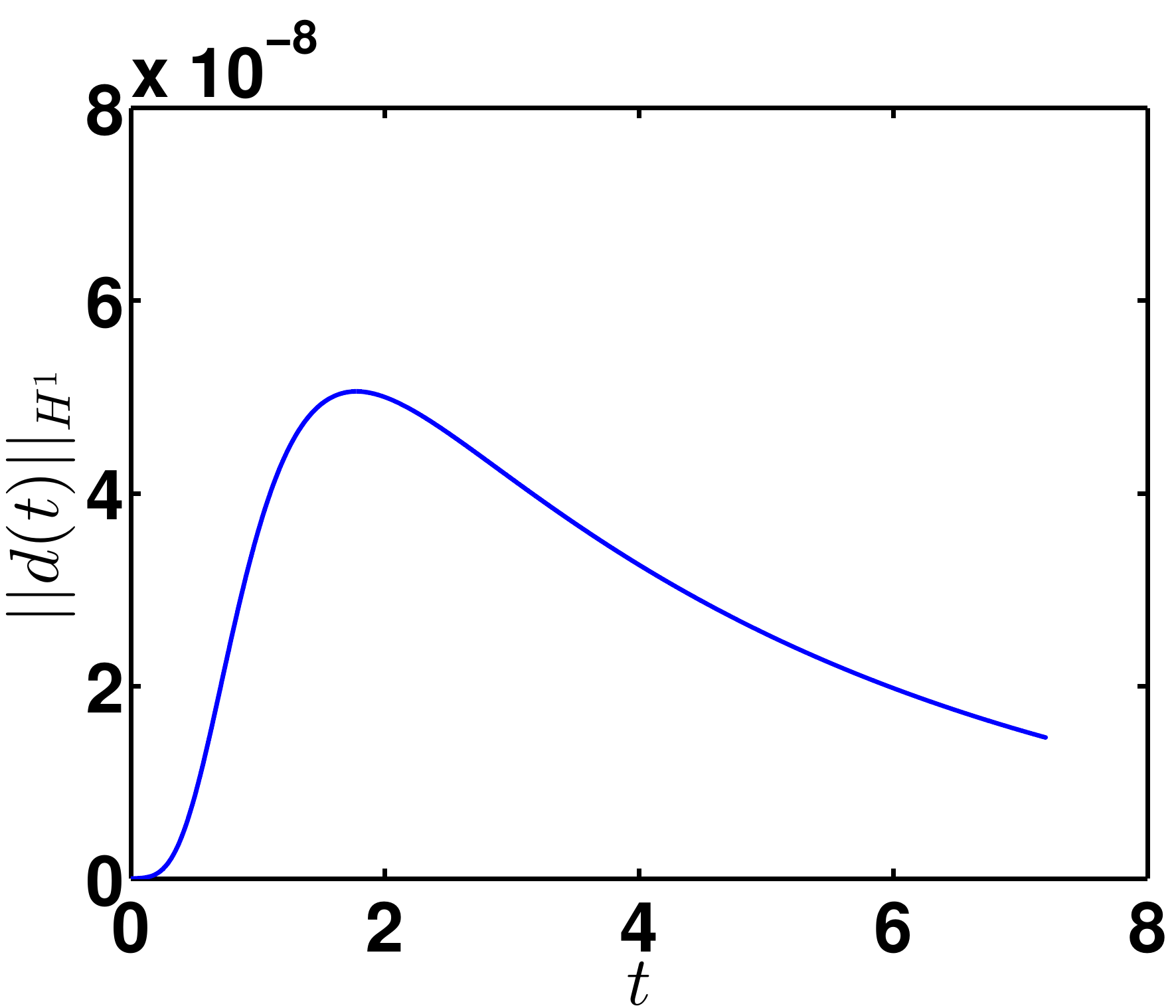}}
	\hspace*{\fill} %
	\\
	\hspace*{\fill} %
	\subfloat[Smallness Method 1]{\includegraphics[width=0.3\textwidth]
	{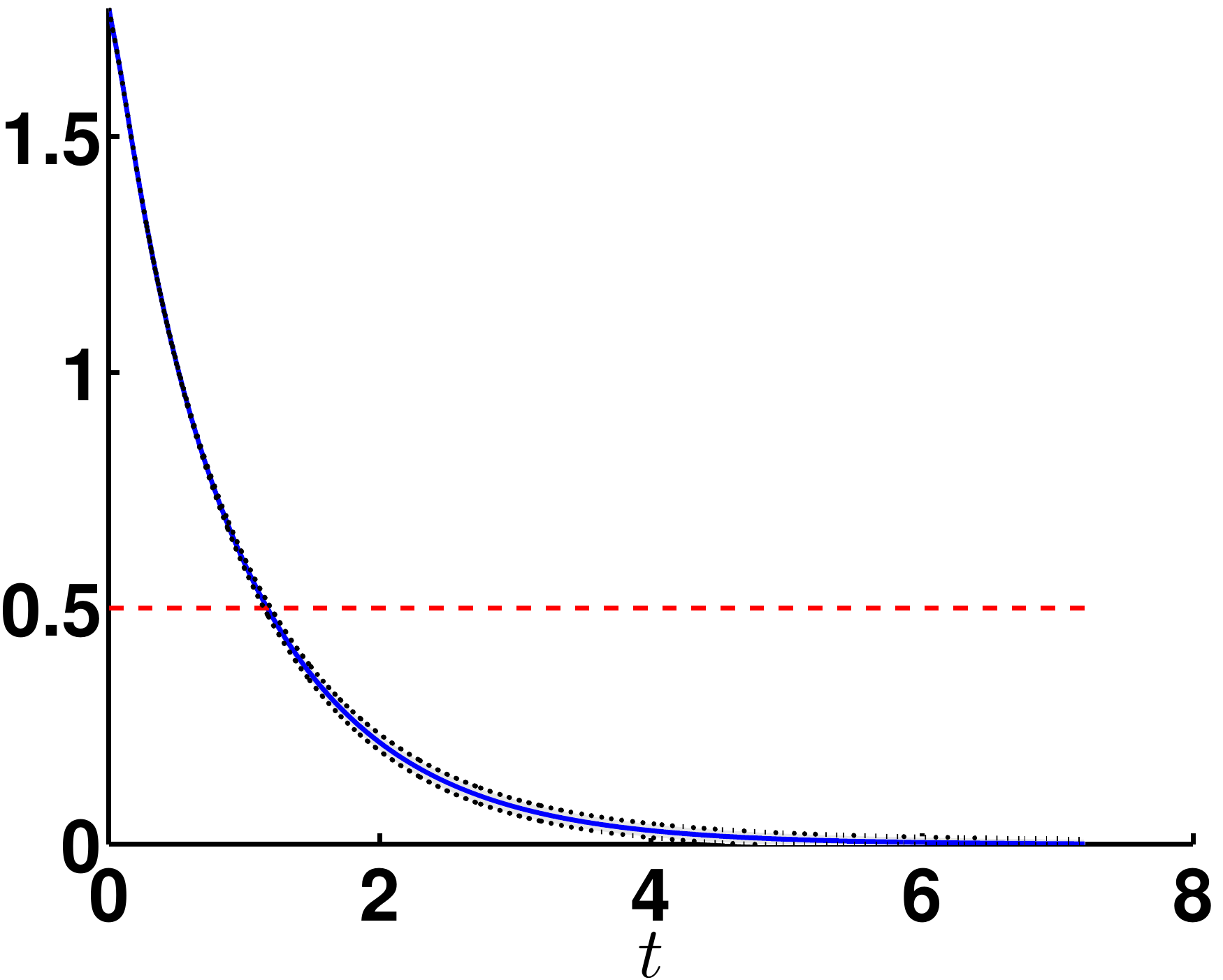}}
	\hspace*{\fill} %
	\subfloat[Smallness Method 2]{\includegraphics[width=0.3\textwidth]
	{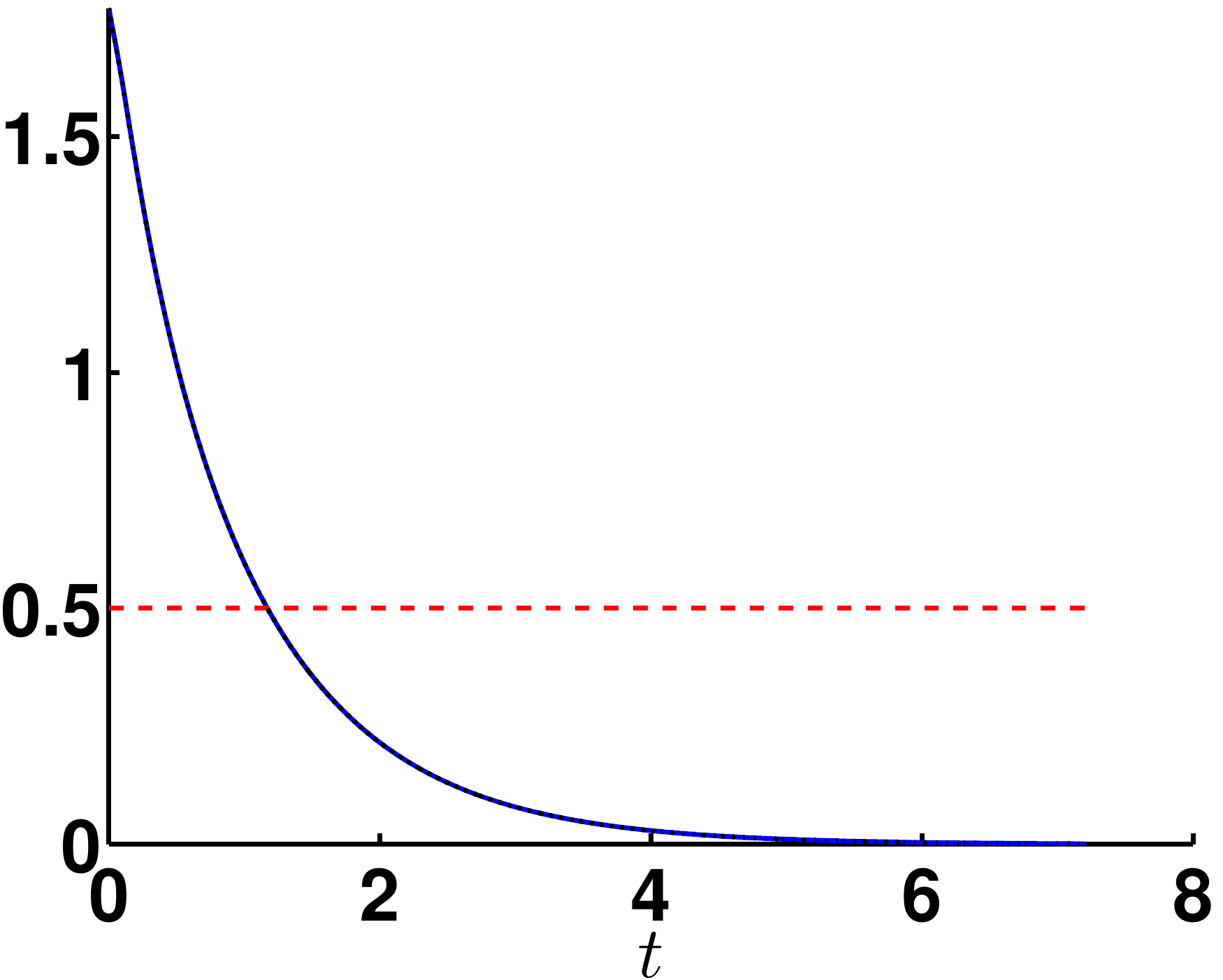}}
	\hspace*{\fill} %
	\subfloat[Smallness Method 3]{\includegraphics[width=0.3\textwidth]
	{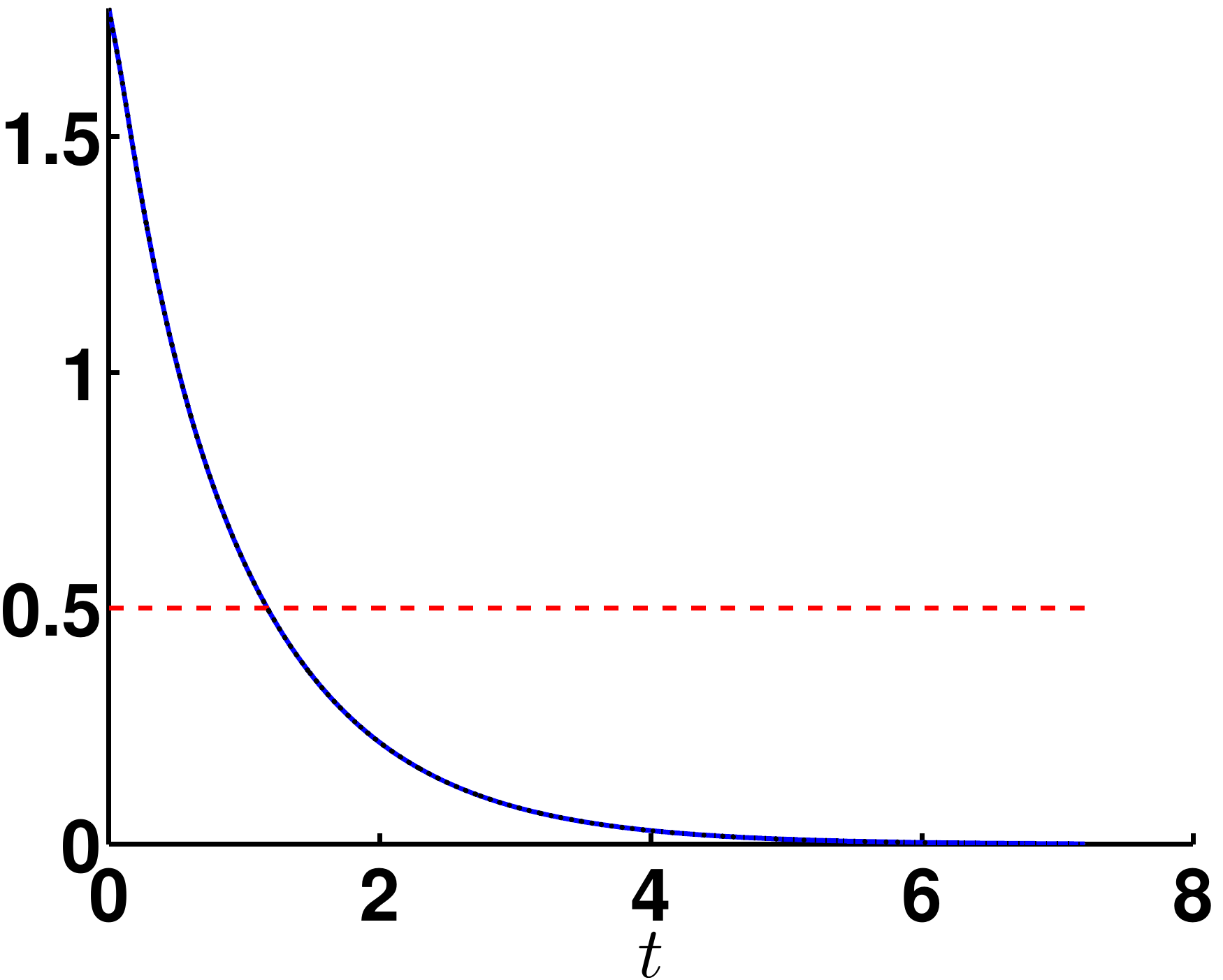}}
	\hspace*{\fill} %
	\\
	\hspace*{\fill} %
	\subfloat[$\| \varphi_{xx} \|_{L^\infty}$]
	{\includegraphics[width=0.3\textwidth]{img/fig1/Linfphi.pdf}}
	\hspace*{\fill} %
	\subfloat[$\|\RES \|_{H^{-1}}$]{\includegraphics[width=0.3\textwidth]
	{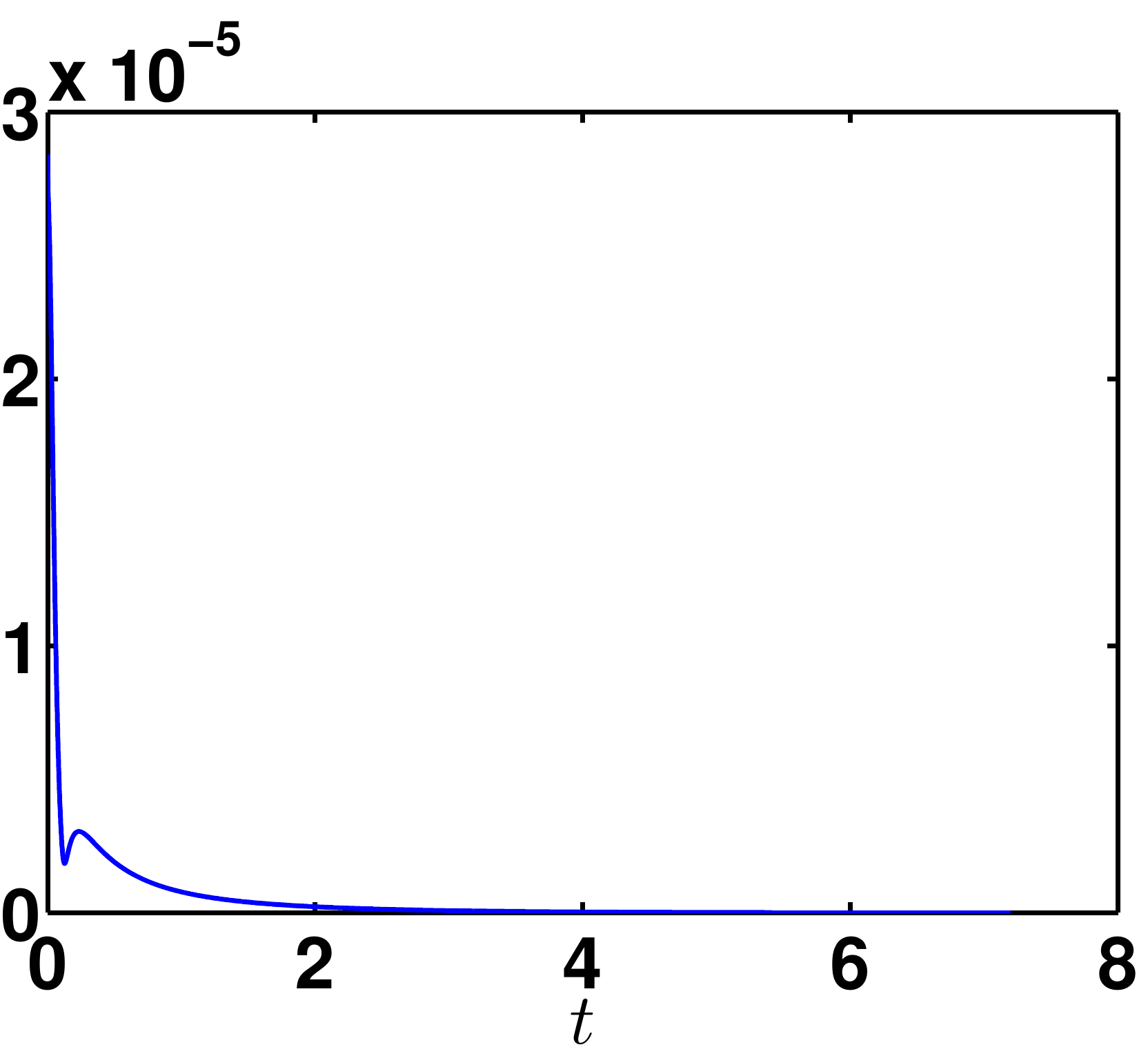}}
	\hspace*{\fill} %
	\subfloat[$\varphi$]{\includegraphics[width=0.3\textwidth]
	{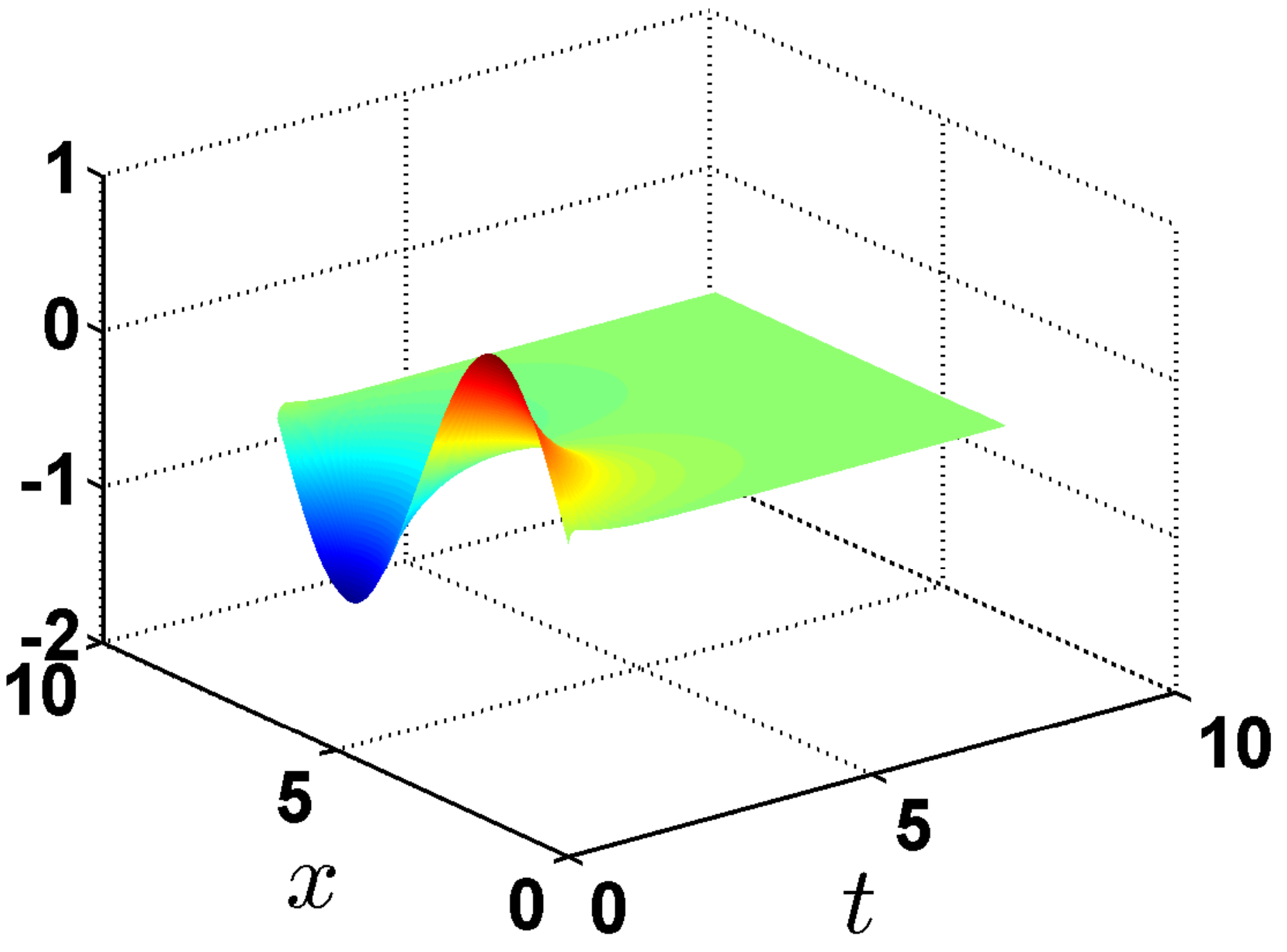}}
	\hspace*{\fill} %
	\\
	\caption{Same setting as in Figure \ref{img:example1}, but now with
	smaller step-size $h=10^{-6}$. Now Method 1 shows global
	existence, too. Note that the order of the Residual decreased
	proportional to our step-size.}
	 \label{img:example2}
	\end{center}
\end{figure}
In Figure \ref{img:example2} we get global existence also with Method 1
by decreasing the stepsize to $h=10^{-6}$. All other parameters
stay unchanged. Note that the order of the residual changed from
$10^{-4}$ to $10^{-5}$ by the same factor we decreased the step-size.

Figure \ref{img:example3} also suggests that Method 1 is inferior to the
other methods. With an already small step-size of $10^{-6}$ Method 1
does not get close to showing global existence, either by 
the smallness condition or the time condition.

\begin{figure}[!ht]%
	\begin{center}
	\hspace*{\fill} %
	\subfloat[Method 1]{\includegraphics[width=0.3\textwidth]
	{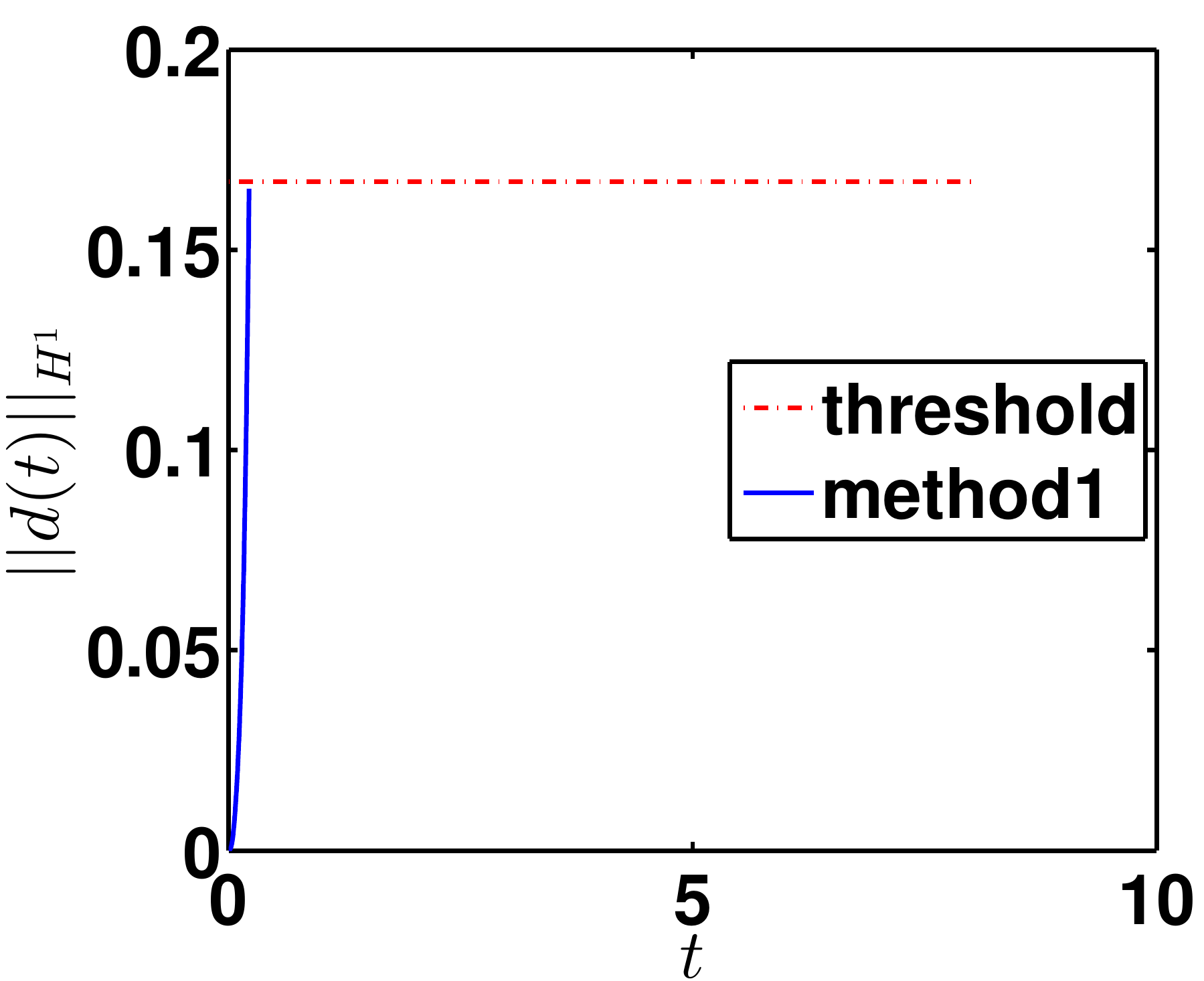}}
	\hspace*{\fill} %
	\subfloat[Method 2]{\includegraphics[width=0.3\textwidth]
	{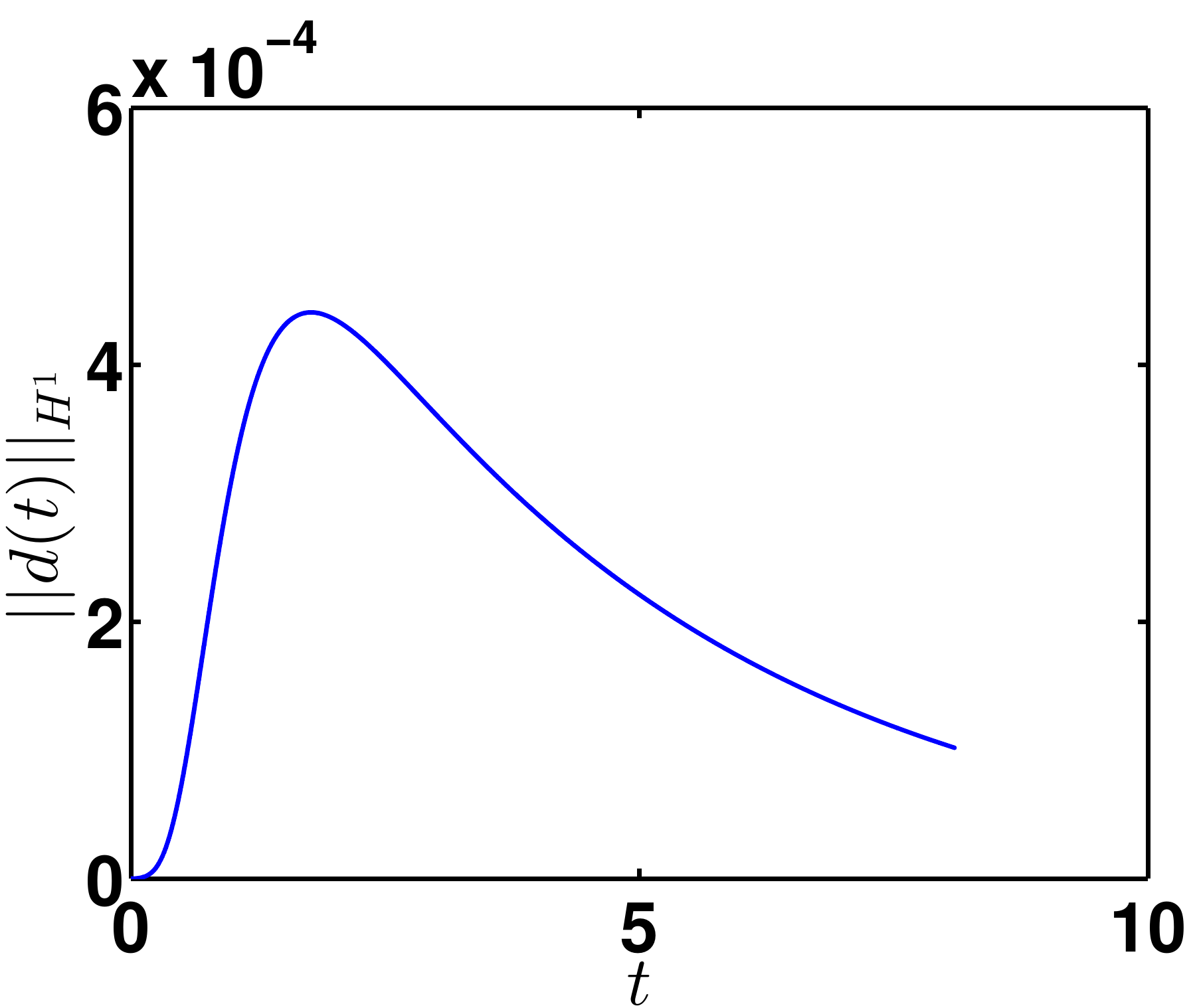}}
	\hspace*{\fill} %
	\subfloat[Method 3]{\includegraphics[width=0.3\textwidth]
	{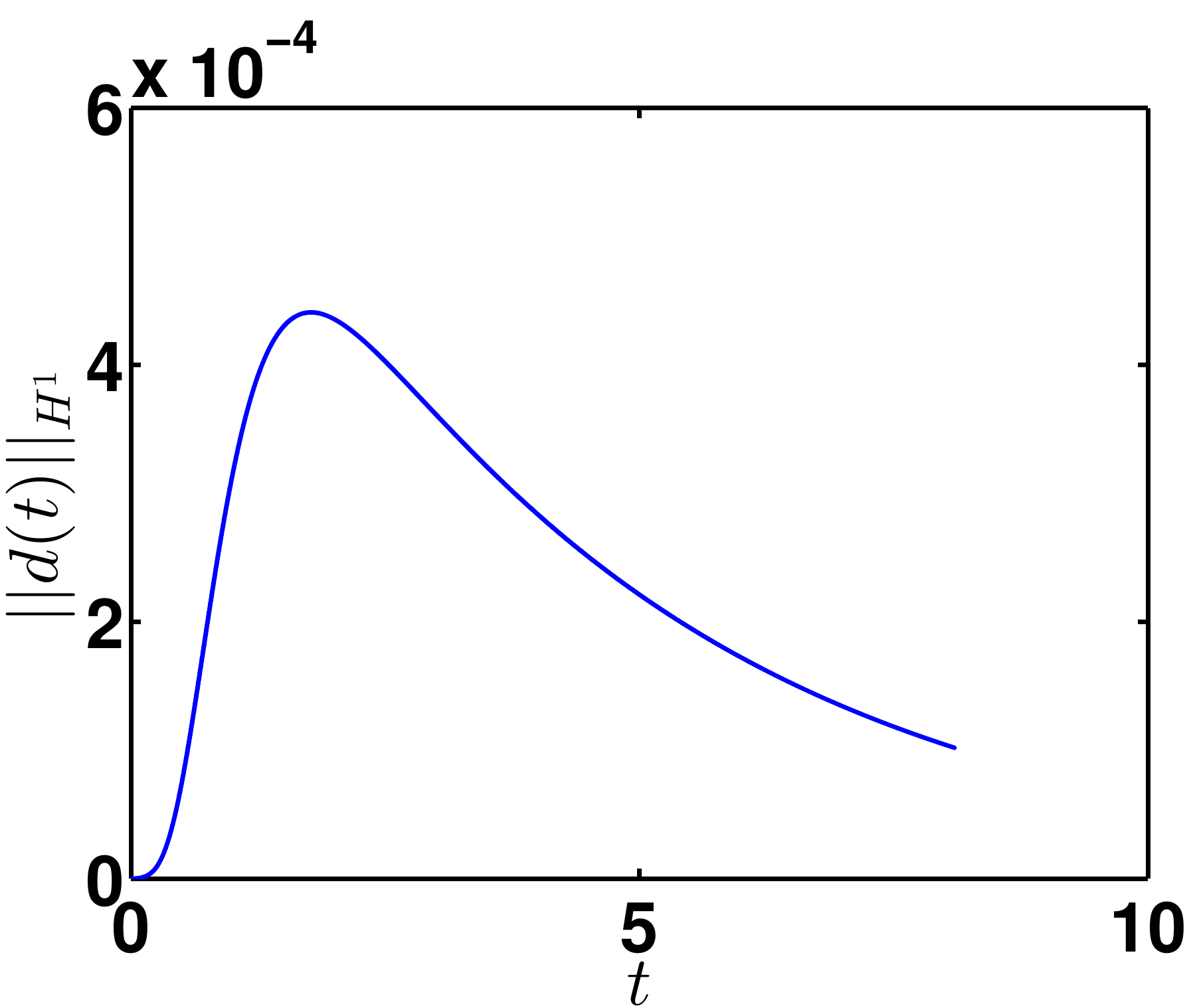}}
	\hspace*{\fill} %
	\\
	\hspace*{\fill} %
	\subfloat[Smallness Method 1]{\includegraphics[width=0.3\textwidth]
	{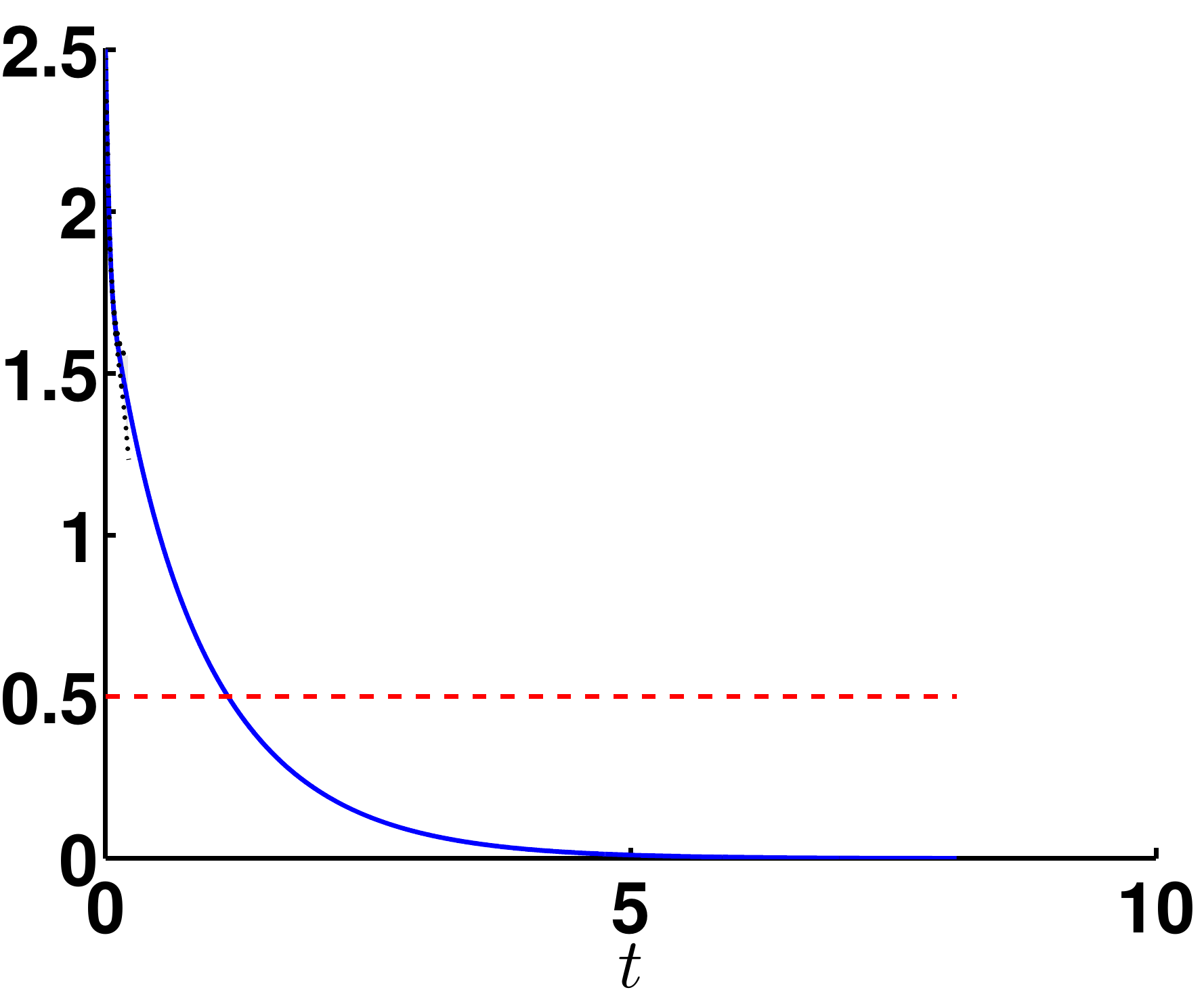}}
	\hspace*{\fill} %
	\subfloat[Smallness Method 2]{\includegraphics[width=0.3\textwidth]
	{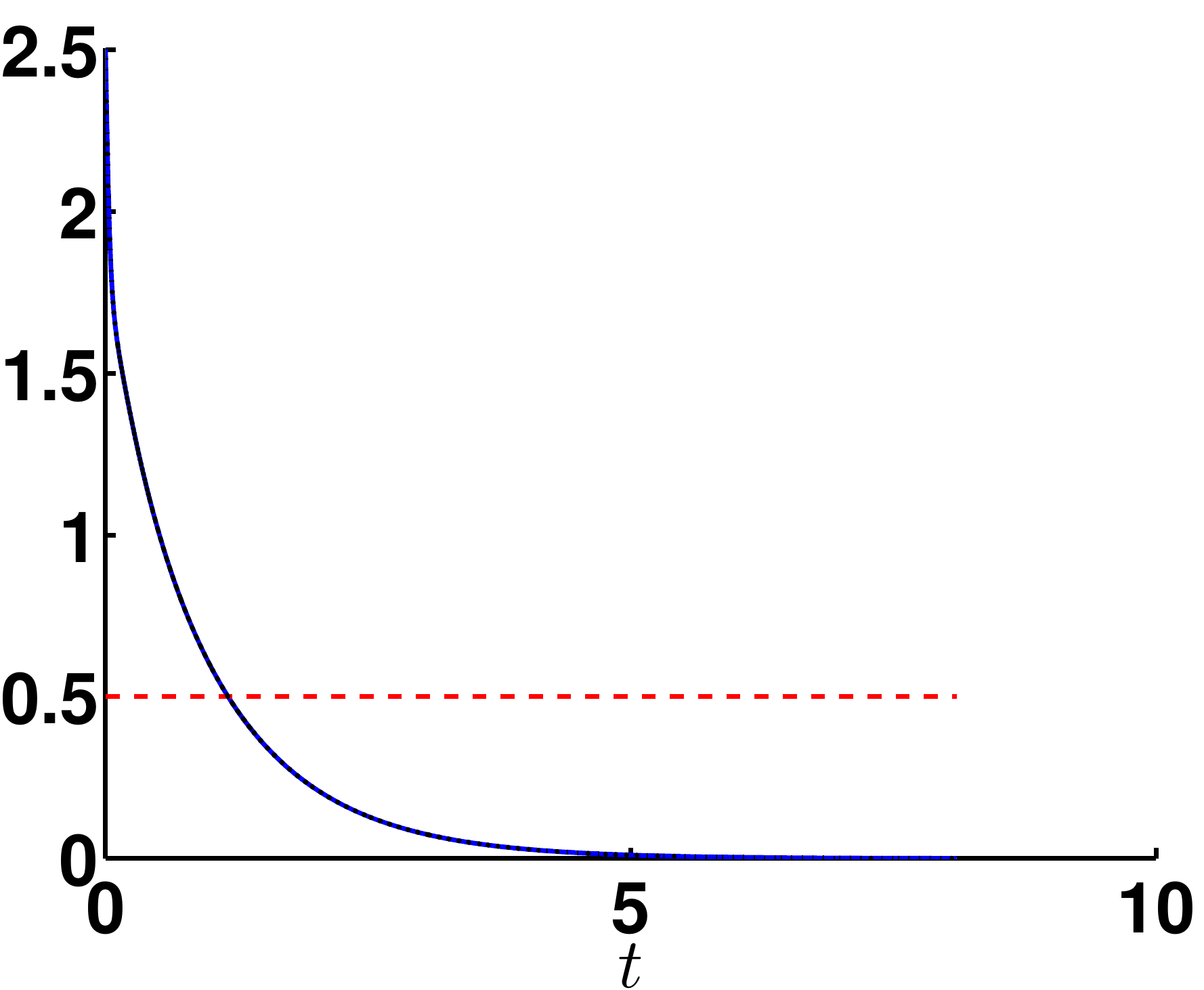}}
	\hspace*{\fill} %
	\subfloat[Smallness Method 3]{\includegraphics[width=0.3\textwidth]
	{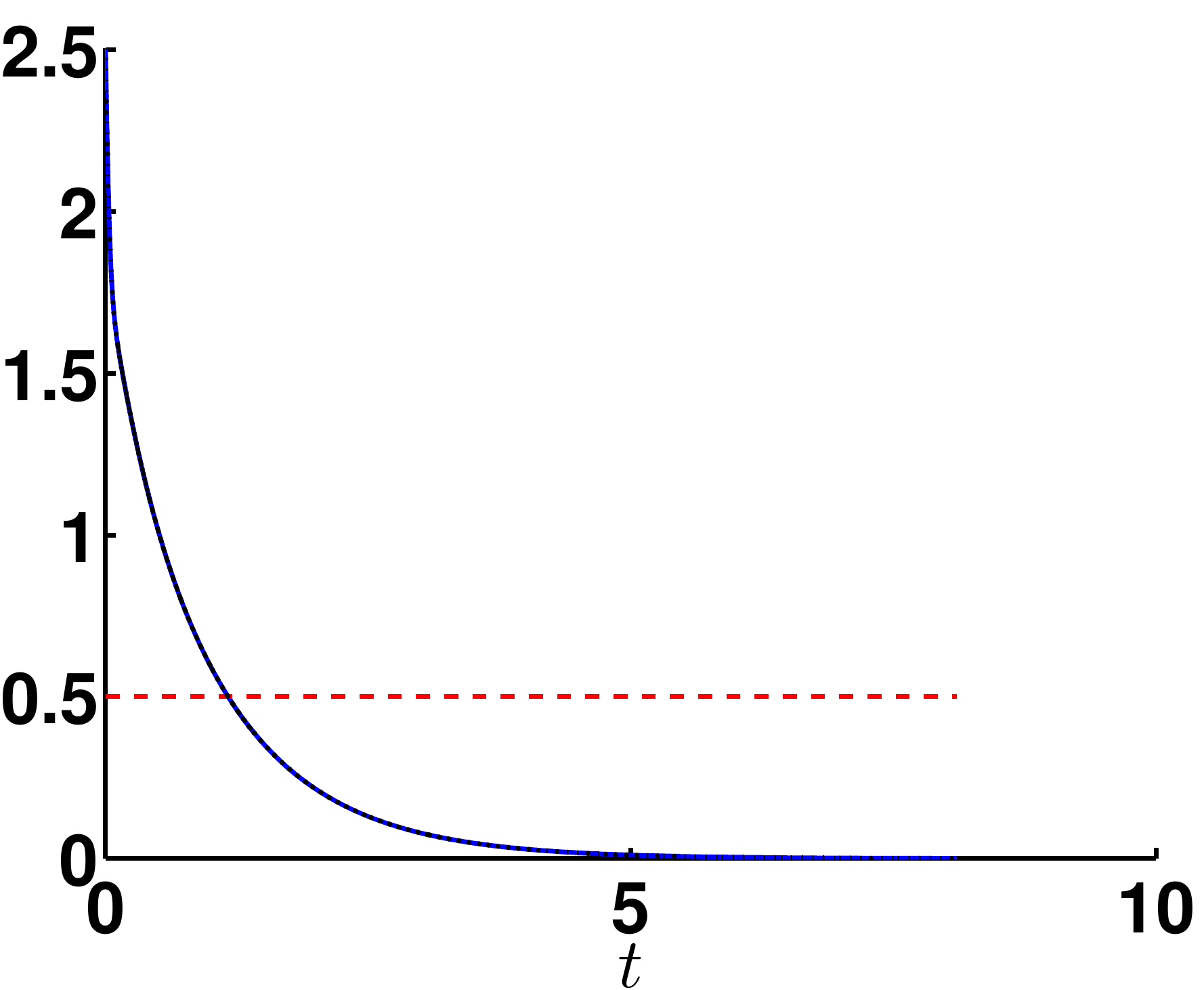}}
	\hspace*{\fill} %
	\\
	\hspace*{\fill} %
	\subfloat[$\| \varphi_{xx} \|_{L^\infty}$]
	{\includegraphics[width=0.3\textwidth]{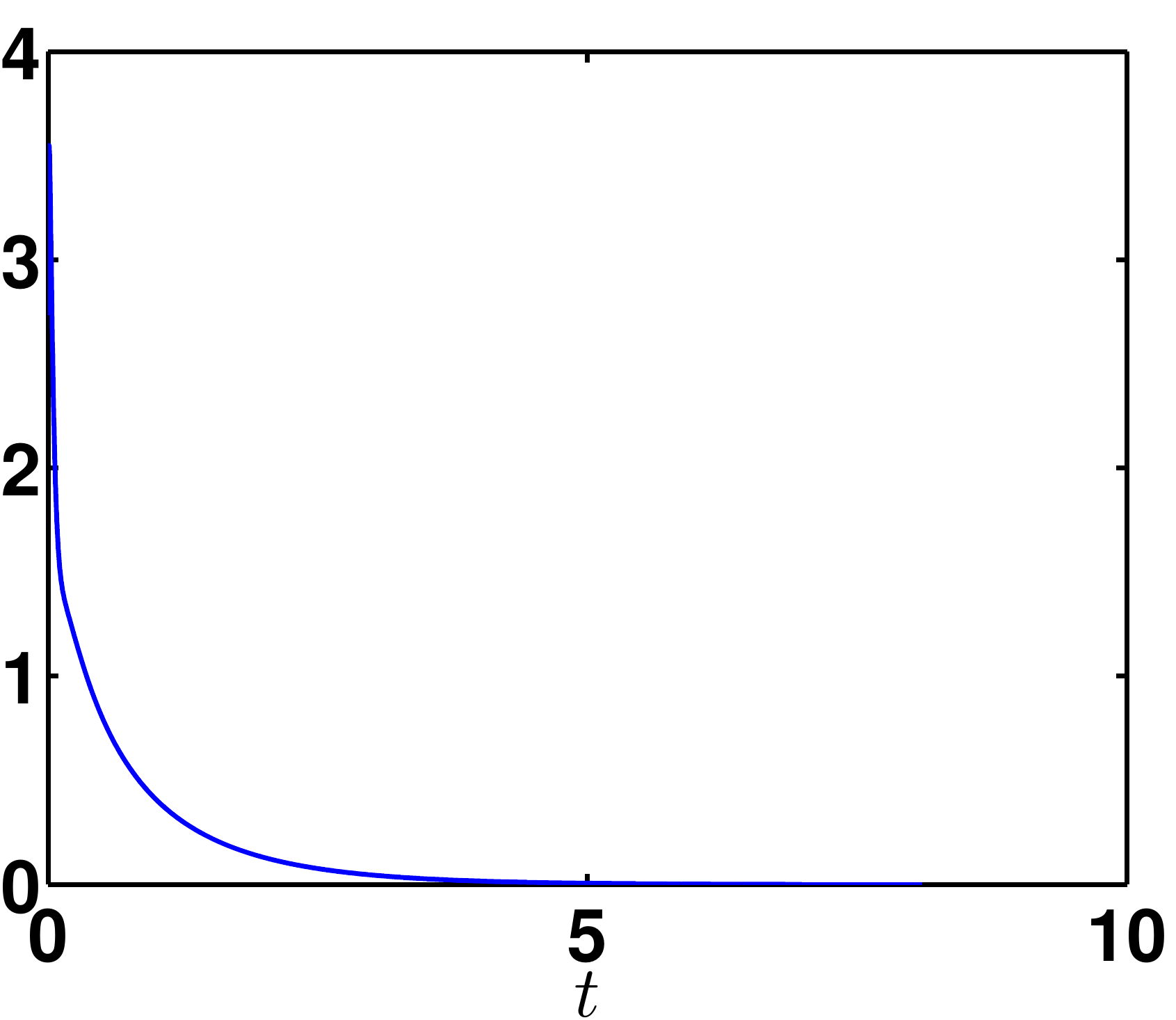}}
	\hspace*{\fill} %
	\subfloat[$\|\RES \|_{H^{-1}}$]{\includegraphics[width=0.3\textwidth]
	{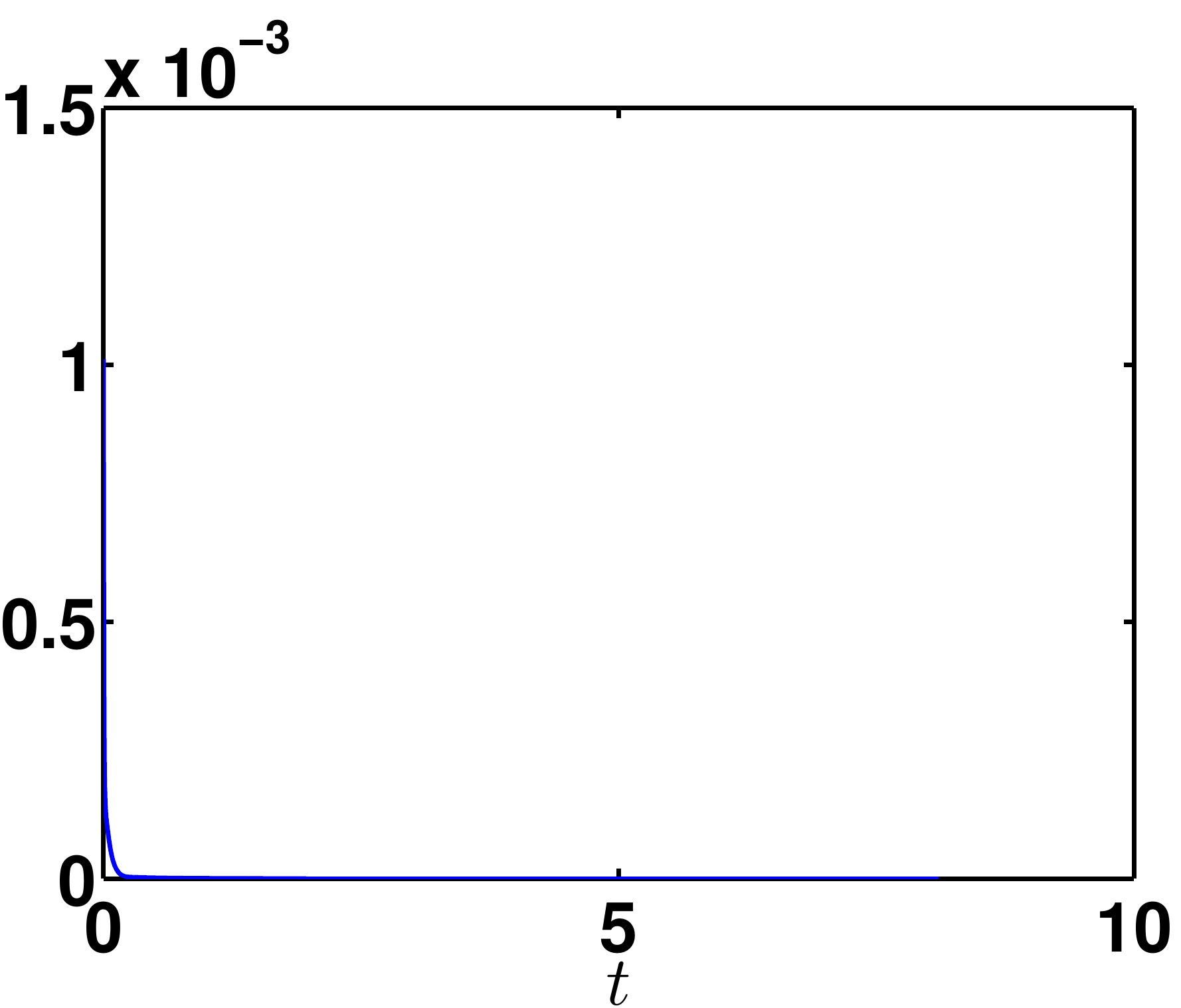}}
	\hspace*{\fill} %
	\subfloat[$\varphi$]{\includegraphics[width=0.3\textwidth]
	{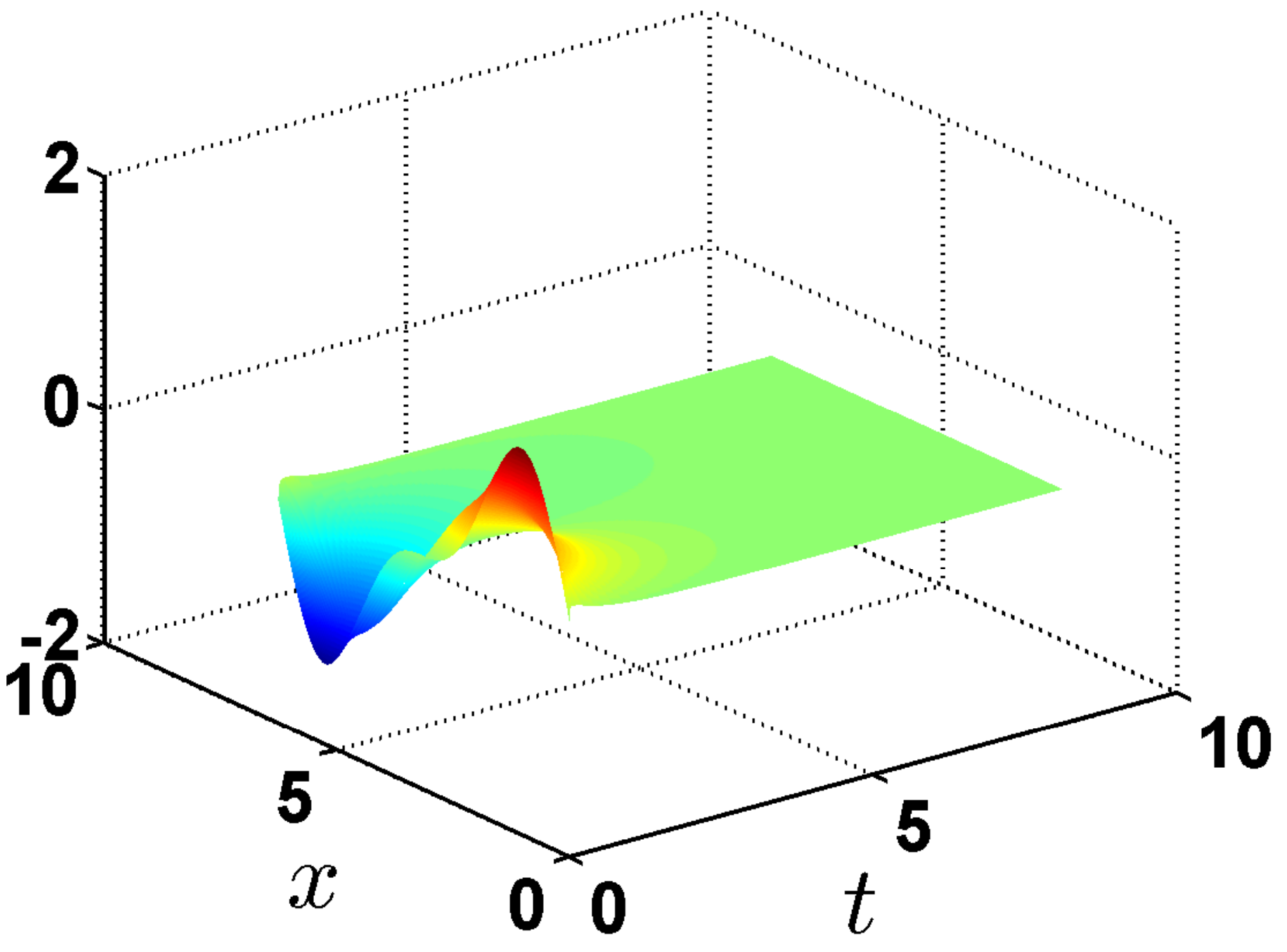}}
	\hspace*{\fill} %
	\\
	\caption{Initial value
	$u(x,0) = \sin(x)+\frac{1}{2}\sin(2x)$, $N=128$
	Fourier modes and step-size $h=10^{-6}$. Method 1 reaches its
	threshold extremely fast. Methods 2 and 3 show global existence by
	both criteria.}
	 \label{img:example3}
	\end{center}
\end{figure}

In Table \ref{table:results} we have collected some more results
for different choices of initial data. One can see that if 
we have global existence by
one criterion, then also by the other. The smallness criterion is
also reached significantly earlier than the time criterion.

\begin{table}[!htb]
\centering
\scalebox{0.8}{
	\begin{tabular}{c|c|c|c|c|c|c|c|c|c}
	\multicolumn{4}{c|}{\;} & \multicolumn{3}{c|}{Smallness}
	& \multicolumn{3}{c}{Time} \\
	$u(x,0)$ & $T^*\approx$ & N & h & M1 & M2 & M3 & M1 & M2 & M3 \\
	\hline
	$\sin(x)$ & $7.1$ & $128$ & $10^{-5}$ & -- & $1.17$ & $1.17$ &
	$1.57$ & \checkmark & \checkmark \\
	\hline
	$\sin(x)$ & $7.1$ & $128$ & $10^{-6}$ & $1.2$ & $1.17$ & $1.17$ &
	\checkmark & \checkmark & \checkmark  \\
	\hline
	$\sin(x)+\frac{1}{2}\sin(2x)$ & $8$ & $128$ & $10^{-6}$ & -- &
	$1.17$  & $1.17$ & $0.22$ & \checkmark & \checkmark  \\
	\hline
	$\cos(x)-\frac{1}{2}\sin(2x)+\frac{1}{3}\cos(3x)$ & $8.3$ &
	$128$ & $10^{-6}$ & -- & $1.2$ & $1.2$ & $0.1$ &
	\checkmark  & \checkmark  \\
	\hline
	$\sin(2x)$ & $7.1$ & $128$ & $10^{-6}$ & -- & $0.12$ & $0.12$ &
	$0.4$  & \checkmark & \checkmark  \\
	\hline
	$2\sin(x)$ & $14.2$ & $128$ & $10^{-6}$ & -- & -- & -- &
	$0.76$ & $0.14$ & $0.14$ \\
	\hline
	$\frac{3}{2}\cos(x)-\frac{1}{2}\sin(2x)+\frac{1}{3}\cos(3x)$ &
	$11.5$ & $128$ & $10^{-6}$ & -- & -- & -- &
	$0.03$ & $0.15$ & $0.15$ \\
	\end{tabular}
}
\caption{All values are rounded to fit into the table. A "--" in the
	"Smallness" columns means, that the smallness criterion was not met by
	the respective method. Else there is the time when it was met. For
	the "Time" columns this turns around. If the time criterion was met,
	the respective method gets a "$\checkmark$", else the time of the
	blowup / reaching the threshold. }
\label{table:results}
\end{table}

As the bounds from Methods 2 and 3 are in all of our simulations
almost identical, we need an artificial example to illustrate the
difference between those methods.
In Figure \ref{img:example4} we do this by artificially setting
a constant, relatively large $\| \RES(t) \|_{H^{-1}}$ and an also
constant, but smaller second derivative
$\| \varphi_{xx} \|_{L^\infty}$, without using any
numerical approximation.
In this case Method 3 delivers the largest time interval as it
stays finite up to $T \approx 0.16$, whereas Method 2 has a blowup at
$T \approx 0.11$ and Method 1 at $T \approx 0.03$.

\begin{figure}[!ht]%
	\begin{center}
	\hspace*{\fill} %
	\subfloat[Method 1]{\includegraphics[width=0.3\textwidth]
	{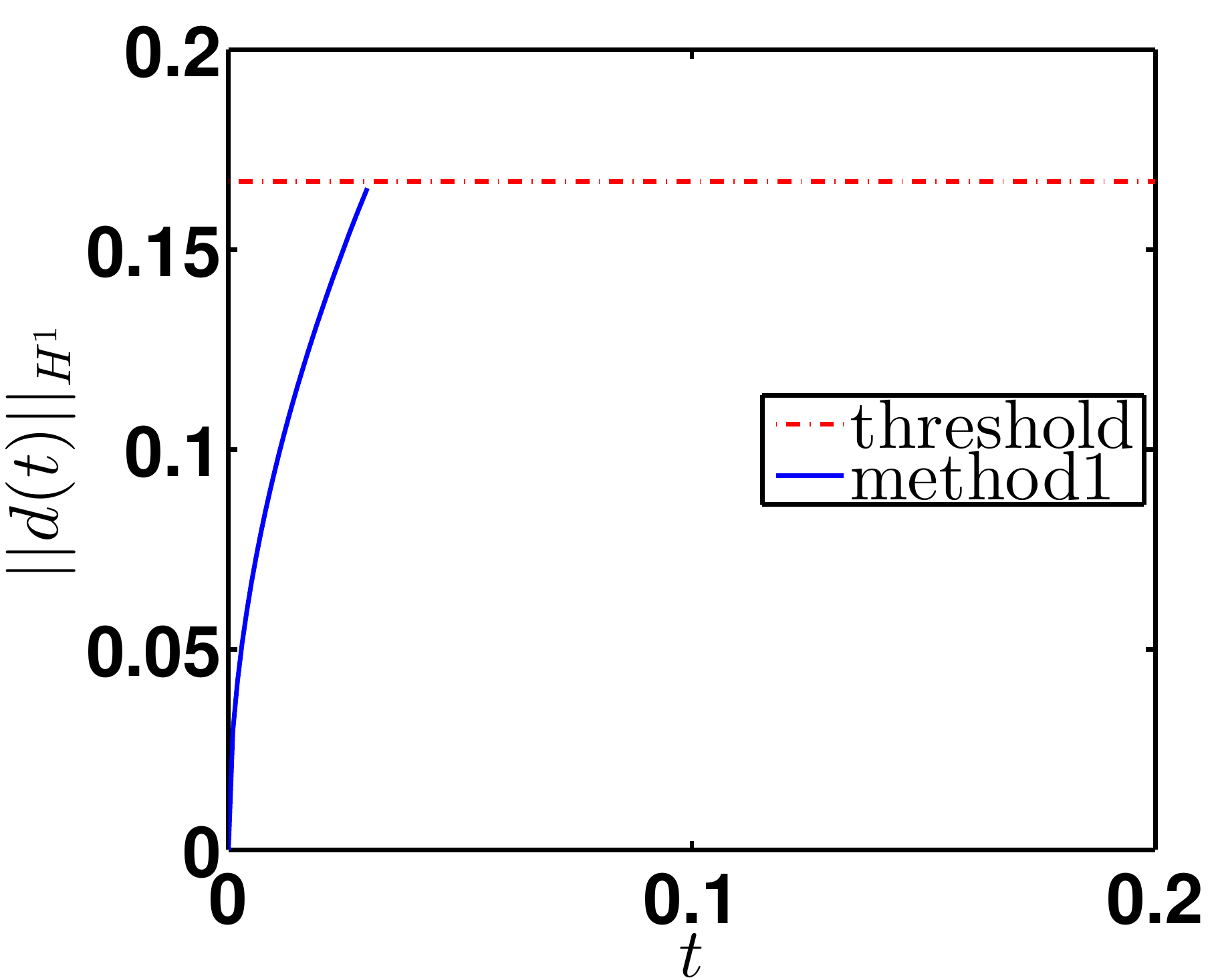}}
	\hspace*{\fill} %
	\subfloat[Method 2]{\includegraphics[width=0.3\textwidth]
	{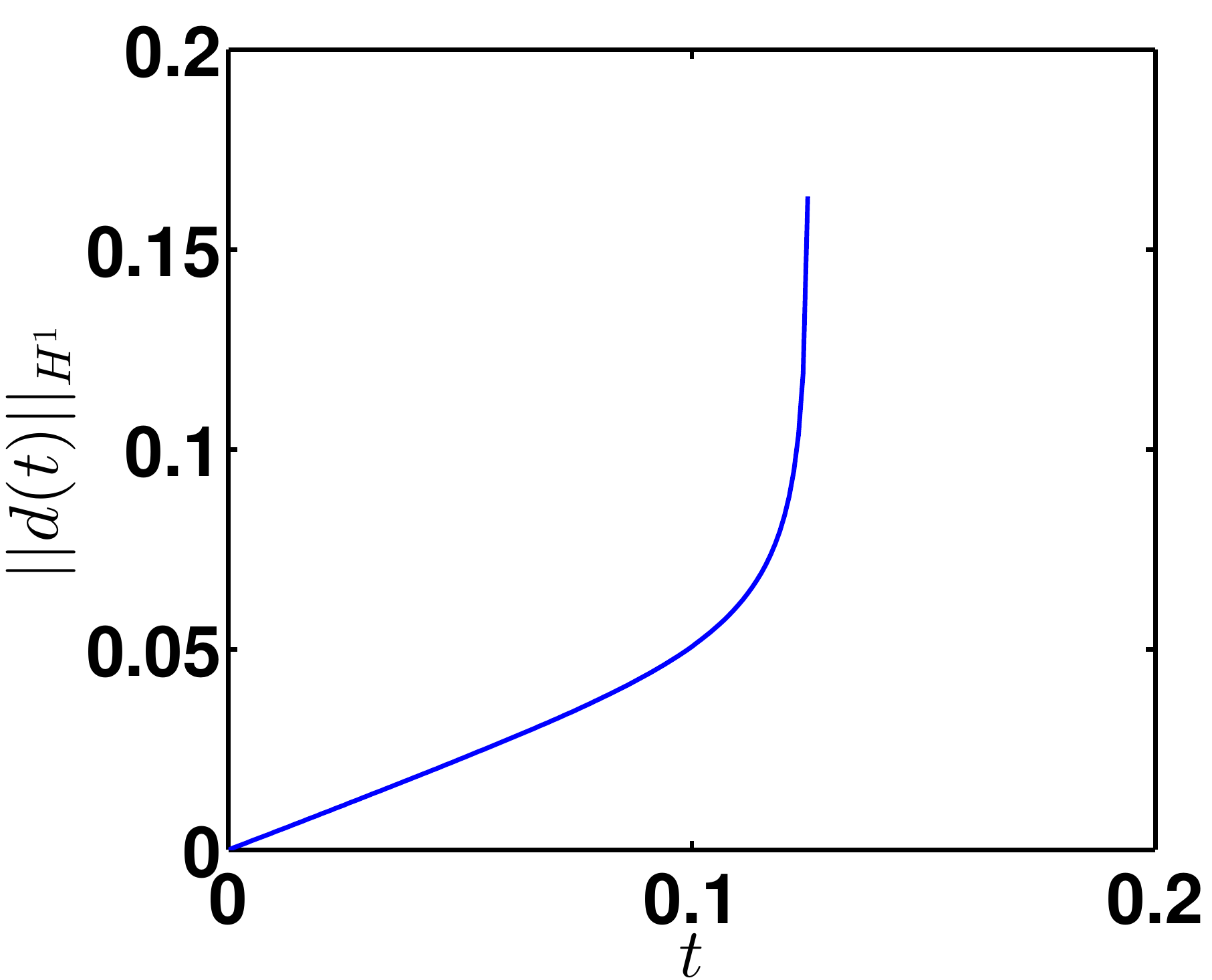}}
	\hspace*{\fill} %
	\subfloat[Method 3]{\includegraphics[width=0.3\textwidth]
	{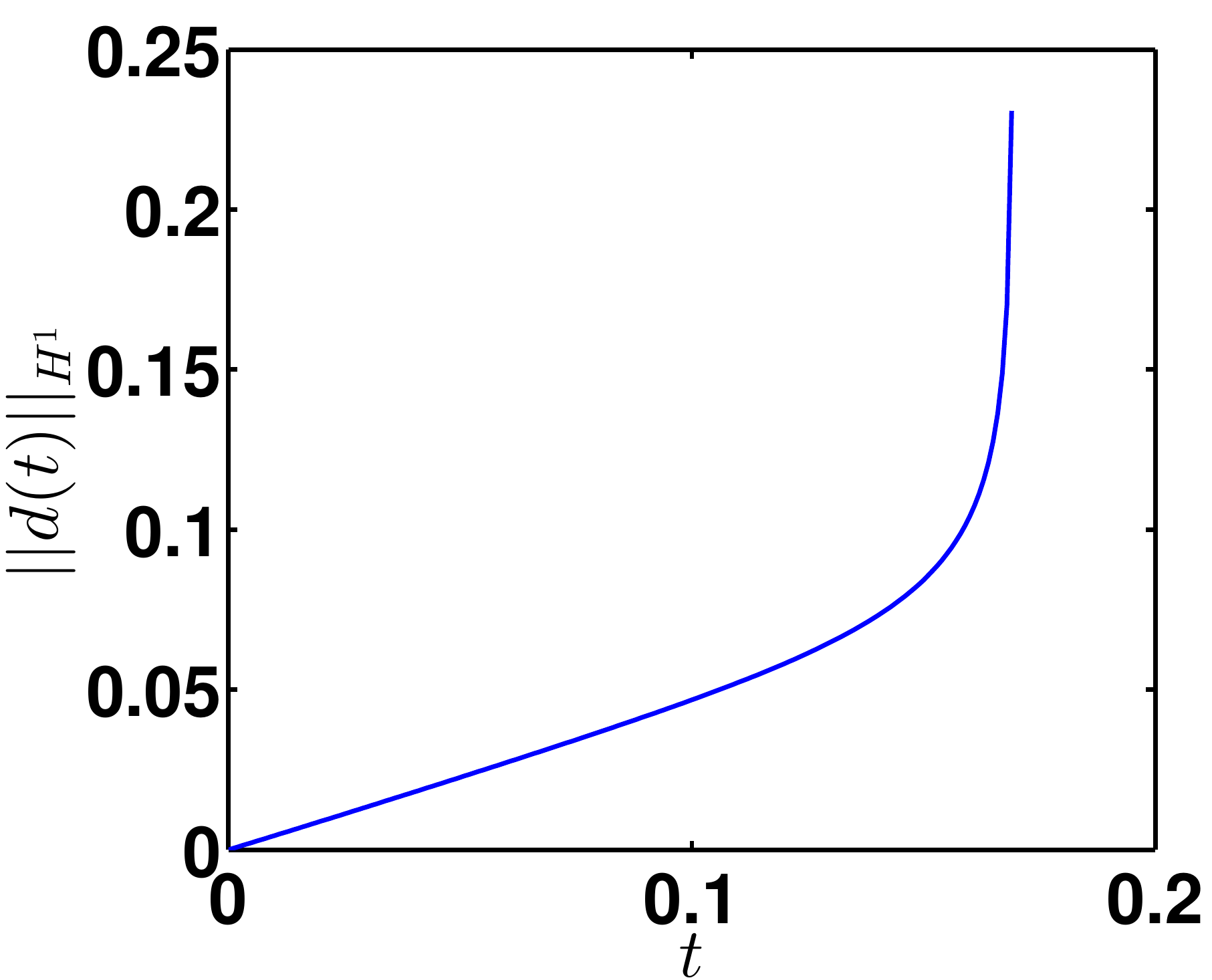}}
	\hspace*{\fill} %
	\\
	\hspace*{\fill} %
	\subfloat[$\| \varphi_{xx} \|_{L^\infty}$]
	{\includegraphics[width=0.3\textwidth]{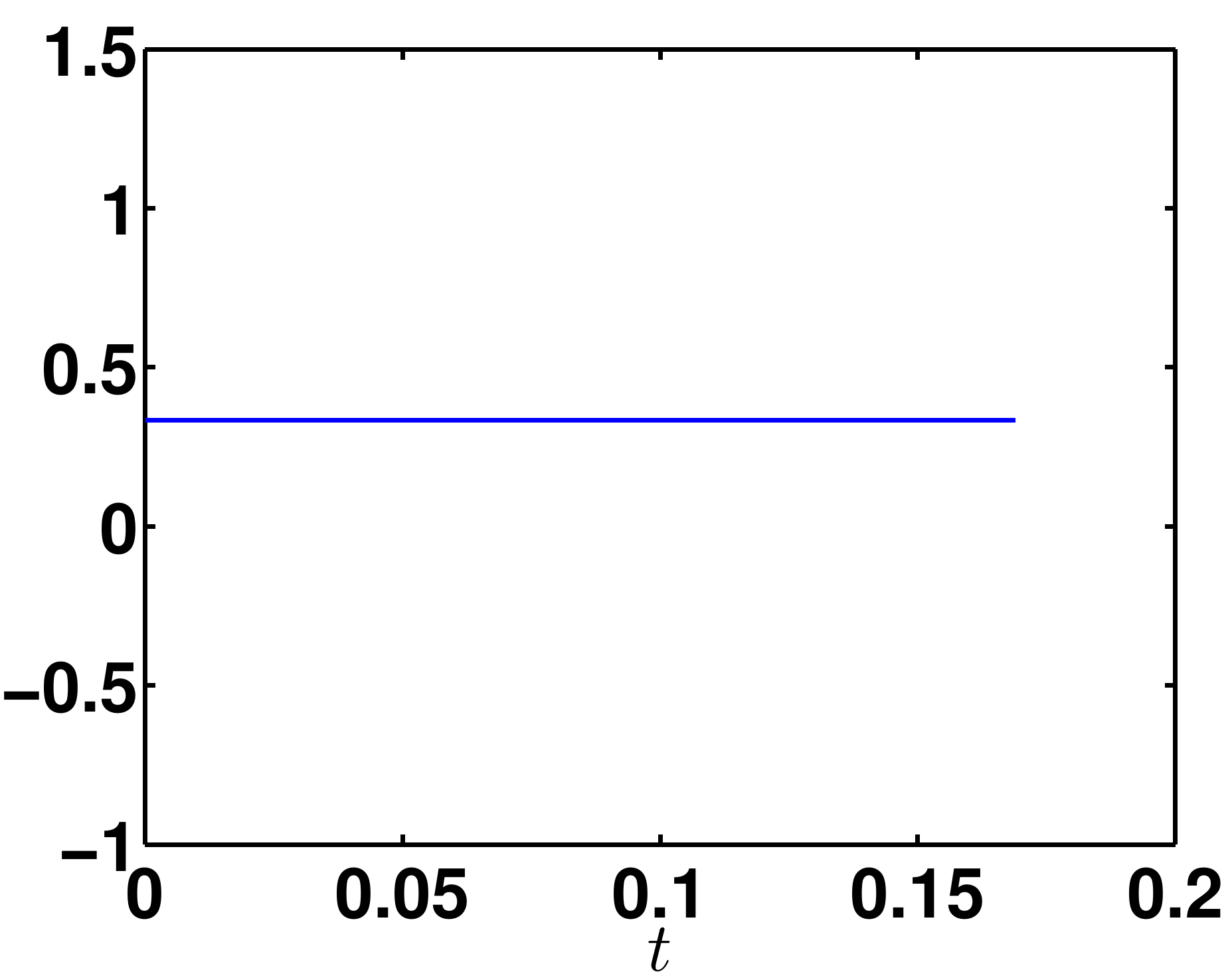}}
	\hspace*{\fill} %
	\subfloat[$\|\RES \|_{H^{-1}}$]{\includegraphics[width=0.3\textwidth]
	{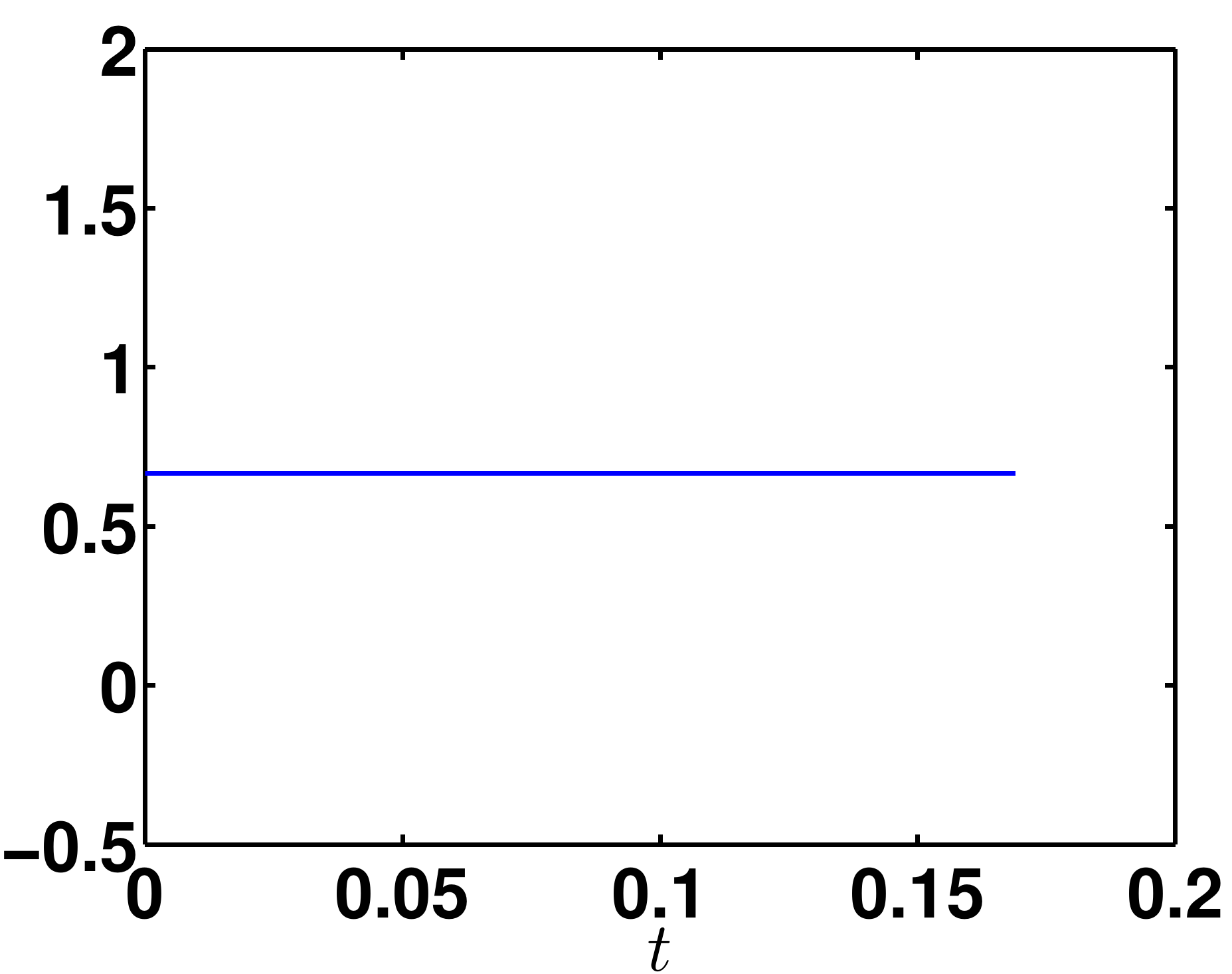}}
	\hspace*{\fill} %
	\\
	\caption{Artificial example for large fixed residual and
		not too large fixed second derivative. Method 3 is superior,
		but all methods do blow up relatively fast,
		as the residual is large.}
	\label{img:example4}
	\end{center}
\end{figure}

\section{Conclusion}

We presented a method to verify global existence and uniqueness
by combining a-posteriori numerical data and a-priori estimates.
Therefore we prove a differential inequality for the error from the
data to the true solution, 
having coefficients depending only on the numerical data.
Three methods are presented to evaluate rigorously  
the analytic upper bounds for the error from these differential
inequalities.

The third method seems to be the best, as it provides rigorous upper bounds
and converges to a solution of the equality in the differential inequality.
Nevertheless, in all practical examples with our Galerkin
approximation, Methods 2 and 3 have shown nearly indistinguishable results.

While our proofs are rigorous, the implementation of the verification
methods are not
completely rigorous since we neglected rounding errors.
Our analysis and computations suggest that
numerical verification of regularity is feasible and can obtain
global existence for initial conditions that are not covered by current
analytical results.

We plan to perform a fully rigorous numerical verification
in a future paper,
using interval arithmetic to keep track of truncation errors.
Moreover, replacing the a-priori estimates of the linearized operator
by rigorous numerical estimates for its spectrum looks promising.



\bibliography{literature}
\bibliographystyle{plain}
\end{document}